\documentclass{amsart}
\usepackage[a4paper]{geometry}
\usepackage{amssymb}
\usepackage{aliascnt}
\usepackage[hyphens]{url}
\usepackage{amsmath}
\usepackage{amsthm}
\usepackage{mathrsfs}
\usepackage[title]{appendix}
\usepackage[english]{babel}
\usepackage{cite}
\usepackage{hyperref}
\usepackage{footmisc}
\usepackage{cleveref}
\usepackage[section]{placeins}
\RequirePackage{textgreek}
\usepackage{enumitem}
\usepackage{bm}
\usepackage{nicefrac}
\usepackage[all,cmtip]{xy}
\usepackage{caption}
\usepackage{float}
\usepackage{mathtools}
\usepackage{xcolor}
\usepackage{xspace}

\hypersetup{
    colorlinks=true,
    linkcolor=blue,
    citecolor=cyan,
    urlcolor=red,
    linktoc=all
}
\numberwithin{equation}{section}

\newtheorem{thm}{Theorem}[section]
\newtheorem*{thm*}{Main Theorem}
\newaliascnt{lem}{thm}
\newtheorem{lem}[lem]{Lemma}
\aliascntresetthe{lem}
\newaliascnt{prop}{thm}
\newtheorem{prop}[prop]{Proposition}
\aliascntresetthe{prop}
\newaliascnt{cor}{thm}

\aliascntresetthe{cor}
\newtheorem*{cor*}{Corollary}
\newaliascnt{obs}{thm}

\aliascntresetthe{obs}
\newaliascnt{defn}{thm}
\newtheorem{defn}[defn]{Definition}
\aliascntresetthe{defn}
\newaliascnt{con}{thm}

\aliascntresetthe{con}
\newaliascnt{que}{thm}

\aliascntresetthe{que}
\newaliascnt{fct}{thm}
\newtheorem{fct}[fct]{Fact}
\aliascntresetthe{fct}
\newaliascnt{claim}{thm}

\aliascntresetthe{claim}
\theoremstyle{remark}
\newaliascnt{rem}{thm}
\newtheorem{rem}[rem]{Remark}
\aliascntresetthe{rem}
\newaliascnt{exm}{thm}
\newtheorem{exm}[exm]{Example}
\aliascntresetthe{exm}

\crefname{thm}{theorem}{theorems}
\Crefname{thm}{Theorem}{Theorems}
\crefname{mainthm}{theorem}{theorems}
\Crefname{mainthm}{Theorem}{Theorems}
\crefname{lem}{lemma}{lemmas}
\Crefname{lem}{Lemma}{Lemmas}
\crefname{prop}{proposition}{propositions}
\Crefname{prop}{Proposition}{Propositions}
\crefname{cor}{corollary}{corollaries}
\Crefname{cor}{Corollary}{Corollaries}
\crefname{obs}{observation}{observations}
\Crefname{obs}{Observation}{Observations}
\crefname{defn}{definition}{definitions}
\Crefname{defn}{Definition}{Definitions}
\crefname{exs}{examples}{examples}
\Crefname{exs}{Examples}{Examples}
\crefname{exm}{example}{examples}
\Crefname{exm}{Example}{Examples}
\crefname{question}{question}{questions}
\Crefname{question}{Question}{Questions}
\crefname{que}{question}{questions}
\Crefname{que}{Question}{Questions}
\crefname{con}{construction}{constructions}
\Crefname{con}{Construction}{Constructions}
\crefname{rem}{remark}{remarks}
\Crefname{rem}{Remark}{Remarks}
\crefname{fct}{fact}{facts}
\Crefname{fct}{Fact}{Facts}
\crefname{claim}{claim}{claims}
\Crefname{claim}{Claim}{Claims}
\crefname{section}{section}{sections}
\Crefname{section}{Section}{Sections}
\crefname{subsection}{section}{sections}
\Crefname{subsection}{Section}{Sections}
\crefname{subsubsection}{section}{sections}
\Crefname{subsubsection}{Section}{Sections}
\crefname{equation}{equation}{equations}
\Crefname{equation}{Equation}{Equations}
\crefname{figure}{figure}{figures}
\Crefname{figure}{Figure}{Figures}
\crefname{table}{table}{tables}
\Crefname{table}{Table}{Tables}
\crefname{enumi}{item}{items}
\Crefname{enumi}{Item}{Items}
\crefname{enumii}{item}{items}
\Crefname{enumii}{Item}{Items}

\newcommand{\pto}{\mathrel{\xrightarrow{\;\mathrm{p}\;}}}
\newcommand{\eqd}{\mathrel{\overset{{\scriptscriptstyle \mathrm{d}}}{=}}}
\newcommand{\eqp}{\mathrel{\overset{{\scriptscriptstyle \mathbb{P}}}{=}}}

\makeatletter
\def\@settitle{%
  \begin{center}%
    \baselineskip18\p@\relax
    \normalfont\Large\scshape
    \@title
  \end{center}%
}
\makeatother

\title[Ergodicity and weak mixing for infinitely divisible stationary processes]{Ergodicity and weak mixing for group-indexed infinitely divisible stationary processes}

\subjclass[2020]{60G10, 60E07, 60G52, 60G55, 60H05, 37A50}

\keywords{Stationary process, infinitely divisible process, ergodicity, weak mixing}

\author[Avraham-Re'em]{Nachi Avraham-Re'em}
\address{Technion Israel Institute of Technology, Faculty of  Mathematics, Technion City, Haifa 3200003, Israel}
\email{nachi.avraham@gmail.com}

\author[Roy]{Emmanuel Roy}
\address{Laboratoire Analyse Géométrie et Applications, Université Paris 13, Institut Galilée,
99 avenue Jean-Baptiste Clément, 93430 Villetaneuse, France}
\email{roy@math.univ-paris13.fr}

\thanks{N.~Avraham-Re'em was supported by the Knut and Alice Wallenberg Foundation (KAW 2021.0258) and the ISF grant No. 3423/24. E.~Roy is a member of project IZES-ANR-22-CE40-0011.}

\begin{document}

\begin{abstract}
We prove that for an arbitrary indexing group, every ergodic infinitely divisible stationary process that is separable in probability is weakly mixing. This shows that, as in the well-known case of Gaussian stationary processes, ergodicity implies weak mixing is intrinsic to infinite divisibility, removing all structural assumptions on the group from prior results. The main ingredient is a general construction of stochastically continuous extensions for separable in probability stationary processes, reducing the problem to stochastically continuous processes indexed by Polish groups and then to countable groups, where we combine the Maruyama representation with an ergodicity criterion for Poisson suspensions.
\end{abstract}

\maketitle

\section{Introduction}

A stationary process indexed by a group \(G\) is a stochastic process \(\left(X_{g}\right)_{g\in G}\) whose law is invariant under left translation by \(G\)'s elements. It is \emph{infinitely divisible} if all finite dimensional distributions are infinitely divisible random vectors. See \cref{sct:preli,sct:necback} for background and for precise definitions of ergodicity and weak mixing.

\smallskip

By the classical L\'{e}vy--Khintchine representation, every infinitely divisible stationary process can be decomposed in law as the sum of two independent stationary processes: a Gaussian stationary process and an infinitely divisible Poissonian stationary process. For Gaussian stationary processes, it is a classical result that ergodicity implies weak mixing. For \(G=\mathbb{Z}\) or \(\mathbb{R}\), it dates back to It{\^o}~\cite{ito194413}, Grenander, Fomin, and Maruyama~\cite[\S5]{maruyama1949harmonic} (see also \cite[\S2]{CornfeldFominSinai1982}). It was extended later to locally compact abelian groups, and ultimately to locally compact groups by Tempelman; see the bibliographical notes in \cite[p.~83]{tempelman2013}, \cite[p.~93]{glasner2003ergodic}, \cite[Appendix~A]{glasnerweiss1997}.

\smallskip

Within the rich class of infinitely divisible Poissonian stationary processes, ergodicity implies weak mixing has been established in various cases under structural assumptions on the indexing group and on the process. We briefly review the existing results below. We remove the assumptions on the group \(G\), and assume only that the process is separable in probability. Note that separability in probability is a property of the law of the process, not of the indexing group, and it is the minimal hypothesis that replaces topological assumptions on \(G\) while ensuring realizability on a standard Borel space (\cref{prop:stand}). It is a standard working assumption in infinitely divisible processes; see e.g. Maruyama \cite{maruyama1970infinitely}, Rajput and Rosinski \cite{rajput1989spectral}, Rosinski \cite{rosinski1995structure}, Samorodnitsky and Taqqu \cite[\S3.11]{samorodnitsky1994stable}, Kabluchko and Stoev \cite{kabluchko2016stochastic}.

\smallskip

Our main theorem demonstrates that, analogous to the Gaussian case, ergodicity implies weak mixing is intrinsic to infinite divisibility, and it remains valid for processes indexed by any group provided the process is separable in probability. This removes all structural assumptions on \(G\) (locally compact, abelian, amenable) and the process being stochastically continuous from previous results surveyed below. Our approach bypasses the ergodic theorems and harmonic analysis central to previous works. Together with the classical Gaussian case and the L\'{e}vy--Khintchine decomposition (via \cref{lem:sam}), this yields the following.

\begin{thm*}
Every ergodic infinitely divisible stationary process that is separable in probability is weakly mixing.
\end{thm*}

\subsubsection*{Previous results}

For Gaussian stationary processes, using the classical results mentioned before, together with taking stochastically continuous extensions, one can deduce that ergodicity implies weak mixing for all separable in probability Gaussian stationary processes indexed by any group.

\smallskip

For infinitely divisible Poissonian stationary processes, many results establishing ergodicity implies weak mixing precede ours. Following the aforementioned classical results about Gaussian stationary processes, the next class in which this property was established comprises stochastically continuous symmetric \(\alpha\)-stable stationary processes with \(\alpha \in (0,2)\). This line of results began with Cambanis, Hardin, and Weron \cite{cambanis1987ergodic}, and Podg\'{o}rski and Weron \cite{PodgorskiWeron1991, Podgorski1992} for \(G=\mathbb{Z}\) or \(\mathbb{R}\). It was subsequently extended by Wang, P. Roy, and Stoev \cite[Thm.~4.1]{Wang2013} to \(G=\mathbb{Z}^d\) or \(\mathbb{R}^d\), and more recently by one of the authors \cite[Thm.~1.1]{AvrahamReem2023} to countable amenable groups \(G\) (see the Notes to Chapter~4 in \cite[pp.~579--582]{samorodnitsky1994stable}). The same was established for symmetric semistable stationary processes with \(G=\mathbb{Z}\) or \(\mathbb{R}\) by Kokoszka and Podg\'{o}rski \cite{kokoszka1992ergodicity}. In a more general setting, and still for \(G=\mathbb{Z}\) or \(\mathbb{R}\) and under stochastic continuity assumption, this was proved for symmetric infinitely divisible stationary processes by Cambanis, Podg\'{o}rski, and Weron \cite[Thm.~2]{cambanis1995chaotic}, and for infinitely divisible processes by Rosi\'{n}ski and \.{Z}ak \cite{rosinskizak1997equivalence}; see also \cite[Thm.~8.5.1]{samorodnitsky2016stochastic}. Recently this was extended to \(G=\mathbb{Z}^d\) and \(\mathbb{R}^d\) by Passeggeri and Veraart \cite[Thm.~4.11]{passeggeri2019mixing}.

\smallskip

All aforementioned results regarding infinitely divisible Poissonian processes were proved under the assumption of the group being locally compact abelian or amenable. These assumptions serve two main purposes in previous works: first, it is assumed that the processes are stochastically continuous, which is defined with respect to a suitable topology on the indexing group, and second, they allow the use of the ergodic theorem both in defining ergodicity, as well as in proofs that rely on the analysis of related functionals, e.g. \emph{dynamical functional} \cite{cambanis1995chaotic}; \emph{codifference} \cite{rosinskizak1997equivalence}; \emph{correlation cascades} \cite{magdziarz2009correlation}.

\subsubsection*{About the Proof}

A main new ingredient is the construction of stochastically continuous extensions (\cref{thm:stochext}), which applies to any separable in probability stationary process and is of independent interest: for any separable in probability stationary \(G\)-process for a group \(G\), we construct a stochastically continuous stationary \(\widehat{G}\)-process, where \(\widehat{G}\) is a Polish group containing a quotient of \(G\) as a dense subgroup. This construction not only preserves the distribution class of the process (e.g. being Gaussian or being infinitely divisible Poissonian), but it also preserves both ergodicity and weak mixing.

\smallskip

Stochastically continuous extensions reduce the problem to stochastically continuous processes indexed by Polish groups, which then reduces the problem to countable dense subgroups (\cref{lem:ergdens}). For countable groups, inspired by the proof of one of the authors in \cite[Thm.~5.8]{roy2007ergodic} for the case of \(G=\mathbb{Z}\), we apply the Maruyama representation and the characterization of ergodicity for Poisson suspensions from \cite[Thm.~2]{avraham2025poissonian}. The key property characterizing ergodicity, from which the weak mixing property follows, is that the measure preserving action on the L\'{e}vy measure space is null: it admits no invariant sets of positive finite measure. This is analogous to Samorodnitsky's characterization \cite[Thm.~3.1]{samorodnitsky2005null} (see also \cite[Thm.~1.3]{AvrahamReem2023}) in symmetric \(\alpha\)-stable stationary processes indexed by \(G=\mathbb{Z}\) or \(\mathbb{R}\), using their Rosinski representation.

\subsubsection*{The structure of the paper}

\Cref{sct:preli} introduces basic preliminaries. \Cref{sct:stochext} clarifies key concepts in stationary processes and their counterparts in ergodic theory, culminating in the construction of stochastically continuous extensions (\cref{thm:stochext}). \Cref{sct:finalproof} provides background on infinite divisibility and Maruyama representation, and the proof of the main theorem. \Cref{sct:exm} presents the null condition for weak mixing, and some applications of the main theorem.

\section{Stationary processes, ergodicity, weak mixing}\label{sct:preli}

\subsection{Stationary \(G\)-processes}

Let \(G\) be a set. A (stochastic) {\bf \(G\)-process} is a collection of random variables \(\mathbf{X}=\left(X_{g}\right)_{g\in G}\) defined on the same probability space \(\left(\Omega,\mathcal{F},\mathbb{P}\right)\), namely \(X_{g}\in L^{0}\left(\Omega,\mathcal{F},\mathbb{P}\right)\) for every \(g\in G\). Technically speaking, each element of \(L^{0}\left(\Omega,\mathcal{F},\mathbb{P}\right)\), a random variable, is defined only \(\mathbb{P}\)-a.s. The {\bf distribution} of \(\mathbf{X}\), denoted \(\mathbb{P}_{\mathbf{X}}\), is the probability measure on \(\left(\mathbb{R}^{G},\mathcal{B}^{G}\right)\), where \(\mathcal{B}\) is the Borel \(\sigma\)-algebra of \(\mathbb{R}\) and \(\mathcal{B}^{G}\) is the product \(\sigma\)-algebra,\footnote{The probability space \(\left(\Omega,\mathcal{F},\mathbb{P}\right)\) is not assumed to be \emph{standard} at this stage. The \(\sigma\)-algebra \(\mathcal{B}^{G}\) is generated by the cylindrical sets, and by the Kolmogorov extension theorem, a probability measure on \(\left(\mathbb{R}^{G},\mathcal{B}^{G}\right)\) is determined uniquely by the finite dimensional distributions.} given by
\[\mathbb{P}_{\mathbf{X}}\big(\left(x_{g}\right)_{g\in G}\in\mathbb{R}^{G}:\left(x_{g_{1}},\dotsc,x_{g_{n}}\right)\in E_{1}\times\dotsm\times E_{n}\big)=\mathbb{P}\left(X_{g_{1}}\in E_{1},\dotsc,X_{g_{n}}\in E_{n}\right).\]
For a \(G\)-process \(\mathbf{X}=\left(X_{g}\right)_{g\in G}\), the (finite dimensional) characteristic functions are the functions
\[\varphi_{\left(g_{1},\dotsc,g_{N}\right)}\left(t_{1},\dotsc,t_{N}\right)\coloneq\mathbb{E}\Big[\exp\Big(i\sum\nolimits_{n=1}^{N}t_{n}X_{g_{n}}\Big)\Big],\]
for \(N\geq 1\) and \(g_{1},\dotsc,g_{N}\in G\). The characteristic functions of \(\mathbf{X}\) determine its distribution. When two \(G\)-processes \(\mathbf{X}\) and \(\mathbf{Y}\) have the same distribution, we will write \(\mathbf{X}\eqd \mathbf{Y}\).

\smallskip

When \(G\) is a group, a \(G\)-process \(\mathbf{X}=\left(X_{g}\right)_{g\in G}\) is called a (left) {\bf stationary \(G\)-process}, if
\[\mathbb{P}\left(X_{g_{1}}\in E_{1},\dotsc,X_{g_{n}}\in E_{n}\right)=\mathbb{P}\left(X_{gg_{1}}\in E_{1},\dotsc,X_{gg_{n}}\in E_{n}\right),\]
for all \(g,g_{1},\dotsc,g_{n}\in G\) and \(E_{1},\dotsc,E_{n}\in\mathcal{B}\). In other words, the \(G\)-processes \(\left(X_{gh}\right)_{h\in G}\) and \(\left(X_{h}\right)_{h\in G}\) are equal in distribution for each \(g\in G\). We will denote the (left) translations of \(G\) on \(\mathbb{R}^{G}\) by \(\mathcal{R}=\{R_{g}\}_{g\in G}\), that is, for \(g\in G\), the transformation \(R_{g}:\mathbb{R}^{G}\to\mathbb{R}^{G}\) is given by
\[R_{g}:\left(x_{h}\right)_{h\in G}\mapsto\left(x_{g^{-1}h}\right)_{h\in G}.\]
Given two stationary \(G\)-processes \(\mathbf{X}\) and \(\mathbf{Y}\), define the stationary \(G\)-processes
\[\mathbf{X}\oplus\mathbf{Y}\coloneq\big(X_{g}^{\prime}+Y_{g}^{\prime}\big)_{g\in G}\quad\text{and}\quad\mathbf{X}\otimes\mathbf{Y}\coloneq\big(X_{g}^{\prime},Y_{g}^{\prime}\big)_{g\in G},\footnote{The coordinates of \(\mathbf{X}\otimes\mathbf{Y}\) are two-dimensional random vectors.}\]
where \(\left(X_{g}^{\prime}\right)_{g\in G}\) and \(\left(Y_{g}^{\prime}\right)_{g\in G}\) are independent, \(\left(X_{g}^{\prime}\right)_{g\in G}\eqd\mathbf{X}\) and \(\big(Y_{g}^{\prime}\big)_{g\in G}\eqd\mathbf{Y}\).

\subsection{Ergodicity and weak mixing}\label{sct:ergwm}

Let \(\left(\mathbb{X},\mathcal{X},\xi\right)\) be a probability space (one should keep in mind \(\left(\mathbb{R}^{G},\mathcal{B}^{G},\mathbb{P}_{\mathbf{X}}\right)\)). A {\bf probability preserving automorphism} of \(\left(\mathbb{X},\mathcal{X},\xi\right)\) is a bi-measurable bijection \(T:\mathbb{X}^{\prime}\to\mathbb{X}^{\prime\prime}\) for some \(\mathbb{X}^{\prime},\mathbb{X}^{\prime\prime}\in\mathcal{X}\) with \(\xi\left(\mathbb{X}^{\prime}\right)=\xi\left(\mathbb{X}^{\prime\prime}\right)=1\), such that \(\xi\left(T^{-1}\left(A\right)\right)=\xi\left(A\right)\) for every \(A\in\mathcal{X}\). Let \(G\) be a group. A {\bf probability preserving \(G\)-system} on \(\left(\mathbb{X},\mathcal{X},\xi\right)\), is a collection \(\mathcal{T}=\{T_{g}\}_{g\in G}\) of probability preserving automorphisms of \(\left(\mathbb{X},\mathcal{X},\xi\right)\), such that
\[T_{e_{G}}=\mathrm{Id}_{\mathbb{X}}\,\,\,\xi\text{-a.s. and }T_{g}\circ T_{h}=T_{gh}\,\,\,\xi\text{-a.s. for all }g,h\in G.\]
The \(G\){\bf-invariant} \(\sigma\){\bf-algebra} of such a probability preserving \(G\)-system \(\mathcal{T}\) is the sub-\(\sigma\)-algebra
\[\mathcal{I}_{\mathcal{T}}\coloneq\{E\in\mathcal{X}:\xi\left(T_{g}^{-1}\left(E\right)\triangle E\right)=0\text{ for every }g\in G\}.\]
For a pair of probability preserving \(G\)-systems \(\mathcal{T}=\{T_{g}\}_{g\in G}\) and \(\mathcal{S}=\{S_{g}\}_{g\in G}\) on \(\left(\mathbb{X},\mathcal{X},\xi\right)\) and \(\left(\mathbb{Y},\mathcal{Y},\eta\right)\), respectively, denote by \(\mathcal{T}\times\mathcal{S}=\{T_{g}\times S_{g}\}_{g\in G}\) the probability preserving \(G\)-system on \(\left(\mathbb{X}\times\mathbb{Y},\mathcal{X}\otimes\mathcal{Y},\xi\otimes\eta\right)\) given by
\[T_{g}\times S_{g}\left(x,y\right)=\left(T_{g}\left(x\right),S_{g}\left(y\right)\right).\]

\smallskip

A probability preserving \(G\)-system \(\mathcal{T}\) is {\bf ergodic} if \(\xi\left(E\right)\in\{0,1\}\) for every \(E\in\mathcal{I}_{\mathcal{T}}\), and it is {\bf weakly mixing} if \(\mathcal{T}\times\mathcal{T}\) is ergodic. It is well-known that weak mixing is equivalent to that \(\mathcal{T}\times\mathcal{S}\) is ergodic for every ergodic probability preserving \(G\)-system \(\mathcal{S}\) (see Appendix~\ref{app:wm}). Denote by \(\widehat{\mathcal{X}}\) the \(\xi\)-completion of \(\mathcal{X}\), namely the \(\sigma\)-algebra of sets that differ from a set in \(\mathcal{X}\) by a set of \(\xi\)-measure \(0\), then the ergodicity of \(\left(\mathbb{X},\mathcal{X},\xi\right)\) is equivalent to that of \(\big(\mathbb{X},\widehat{\mathcal{X}},\xi\big)\), and the same is true for weak mixing. In fact, the invariant \(\sigma\)-algebra of \(\big(\mathbb{X},\widehat{\mathcal{X}},\xi\big)\) is \(\widehat{\mathcal{I}_{\mathcal{T}}}\), the \(\xi\)-completion of \(\mathcal{I}_{\mathcal{T}}\).

\begin{rem}
Ergodicity and weak mixing admit various alternative characterizations when additional structure is imposed. Most notably, when \(G\) is a countable group or a locally compact second countable group, and the system comprises a measurable action on a standard probability space (that is, a \emph{probability preserving action}; see the definition below), the associated Koopman representation yields a weakly continuous unitary representation of \(G\) on a separable Hilbert space, which can be used to define ergodicity and weak mixing. When \(G\) is further amenable, then ergodicity and weak mixing can be formulated entirely using the ergodic theorem, as is customary in the probabilistic literature. See \cite[Ch.~2 and Theorem~1.6]{tempelman2013}.
\end{rem}

Given a stationary \(G\)-process \(\mathbf{X}\), one naturally considers the probability preserving \(G\)-system
\[\mathcal{R}=\{R_{g}\}_{g\in G}\text{ on }\left(\mathbb{R}^{G},\mathcal{B}^{G},\mathbb{P}_{\mathbf{X}}\right),\]
for the aforementioned translations given by \(R_{g}:\left(x_{h}\right)_{h\in G}\mapsto\left(x_{g^{-1}h}\right)_{h\in G}\) for every \(g\in G\). The stationarity of \(\mathbf{X}\) means that \(\mathcal{R}\) indeed forms a probability preserving \(G\)-system with respect to \(\mathbb{P}_{\mathbf{X}}\). Note that the probability preserving \(G\)-system associated with \(\mathbf{X}\otimes\mathbf{X}\) is nothing but \(\mathcal{R}\times\mathcal{R}\).

\begin{defn}
Let \(G\) be a group and \(\mathbf{X}\) a stationary \(G\)-process.
\begin{itemize}
    \item \(\mathbf{X}\) is called {\bf ergodic} (or {\bf \(G\)-ergodic}) if its associated probability-preserving \(G\)-system \(\mathcal{R}\) on \(\left(\mathbb{R}^{G},\mathcal{B}^{G},\mathbb{P}_{\mathbf{X}}\right)\) is ergodic.
    \item \(\mathbf{X}\) is called {\bf weakly mixing} (or {\bf \(G\)-weakly mixing}) if \(\mathbf{X}\otimes\mathbf{X}\) is ergodic. Equivalently, if its associated probability preserving \(G\)-system \(\mathcal{R}\) on \(\left(\mathbb{R}^{G},\mathcal{B}^{G},\mathbb{P}_{\mathbf{X}}\right)\) is weakly mixing.
\end{itemize}
\end{defn}

Both ergodicity and weak mixing depend only on the distribution of the process. Weak mixing always implies ergodicity, but the converse generally fails. The following lemma will be useful in determining the ergodicity and weak mixing in infinitely divisible processes, as by virtue of the L\'{e}vy--Khintchine representation their characteristic functions never vanish. It was observed by Samorodnitsky in his proof of \cite[Thm.~3.1]{samorodnitsky2005null}, and we include a proof here.

\begin{lem}[Samorodnitsky]\label{lem:sam}
Let \(G\) be a group, and let \(\mathbf{X}\) and \(\mathbf{Y}\) be independent stationary \(G\)-processes whose finite dimensional characteristic functions never vanish. If \(\mathbf{X}\oplus\mathbf{Y}\) is ergodic then both \(\mathbf{X}\) and \(\mathbf{Y}\) are ergodic.
\end{lem}

\begin{proof}[Proof of \cref{lem:sam}]
Suppose one of \(\mathbf{X}\) and \(\mathbf{Y}\) is not ergodic, say \(\mathbf{X}\). Then there is \(E\in\mathcal{B}^{G}\) with \(0<\mathbb{P}_{\mathbf{X}}\left(E\right)<1\) and \(\mathbb{P}_{\mathbf{X}}\left(R_{g}^{-1}\left(E\right)\triangle E\right)=0\) for every \(g\in G\). Using the independence of \(\mathbf{X}\) and \(\mathbf{Y}\), decompose
\begin{equation}\label{eq:conv}
\mathbb{P}_{\mathbf{X}\oplus\mathbf{Y}}=\mathbb{P}_{\mathbf{X}}\left(E\right)\cdot\mathbb{P}_{\left(\mathbf{X}\mid E\right)\oplus\mathbf{Y}}+\left(1-\mathbb{P}_{\mathbf{X}}\left(E\right)\right)\cdot\mathbb{P}_{\left(\mathbf{X}\mid E^{\complement}\right)\oplus\mathbf{Y}},
\end{equation}
where \(\mathbf{X}\mid E\) and \(\mathbf{X}\mid E^{\complement}\) stand for independent \(G\)-processes whose distributions are the same as that of \(\mathbf{X}\) conditioned on \(E\) and \(E^{\complement}\), respectively. By the invariance of \(E\), both \(\mathbf{X}\mid E\) and \(\mathbf{X}\mid E^{\complement}\) are stationary \(G\)-processes. Additionally, since the finite dimensional characteristic functions of \(\mathbf{Y}\) never vanish, \(\left(\mathbf{X}\mid E\right)\oplus\mathbf{Y}\) and \(\big(\mathbf{X}\mid E^{\complement}\big)\oplus\mathbf{Y}\) cannot be equal in distribution, thus \(\mathbb{P}_{\left(\mathbf{X}\mid E\right)\oplus\mathbf{Y}}\neq\mathbb{P}_{\left(\mathbf{X}\mid E^{\complement}\right)\oplus\mathbf{Y}}\). Consequently, the decomposition \eqref{eq:conv} expresses \(\mathbb{P}_{\mathbf{X}\oplus\mathbf{Y}}\) as a nontrivial convex combination of stationary \(G\)-processes, hence \(\mathbb{P}_{\mathbf{X}\oplus\mathbf{Y}}\) is not ergodic.\footnote{For a proof of this standard fact, see \cite[Thm.~4.4]{einsiedler2011} (while stated for one automorphism of a compact metric space, the same proof is valid in general).}
\end{proof}

\section{Stochastically continuous extensions}\label{sct:stochext}

The following result allows the reduction of the general case in the main theorem to the setting of Polish groups and stochastically continuous stationary processes. This theorem is motivated by an observation we first heard from A. I. Danilenko, namely that the ergodicity of a probability preserving action of an arbitrary group remains unchanged under closure of the image of the group within the automorphism group.

\smallskip

Recall that a {\bf Polish group} is a topological group whose topology is induced from a complete separable metric. When \(G\) is a Polish group, a \(G\)-process \(\mathbf{X}=\left(X_{g}\right)_{g\in G}\) is {\bf stochastically continuous} (or {\bf continuous in probability}) if
\[X_{g}\pto X_{g_{o}}\quad\text{as }g\to g_{o}\text{ for every }g_{o}\in G.\]
Here and later on, \(\pto\) stands for convergence in probability.

\begin{thm}\label{thm:stochext}
Let \(G\) be a group and \(\mathbf{X}=\left(X_{g}\right)_{g\in G}\) a separable in probability stationary \(G\)-process. Then \(\mathbf{X}\) admits a {\bf stochastically continuous extension}: a Polish group \(\widehat{G}\) with a homomorphism \(\tau:G\to\widehat{G}\), and a stationary \(\widehat{G}\)-process \(\widehat{\mathbf{X}}=\big(\widehat{X}_{T}\big)_{T\in\widehat{G}}\), such that:
\begin{enumerate}
    \item \(\tau\left(G\right)\leq\widehat{G}\) is a dense subgroup.
    \item \(\widehat{\mathbf{X}}\) is stochastically continuous.
    \item \(\widehat{\mathbf{X}}\mid_{\tau\left(G\right)}\coloneq\big(\widehat{X}_{\tau\left(g\right)}\big)_{g\in G}\eqd \mathbf{X}\).
\end{enumerate}
Furthermore, whenever \(\widehat{\mathbf{X}}\) is a stochastically continuous extension of \(\mathbf{X}\), one has:
\begin{enumerate}[label=(\roman*), align=left, labelsep=.5em]
    \item \(\mathbf{X}\) is \(G\)-ergodic \(\iff\) \(\widehat{\mathbf{X}}\) is \(\widehat{G}\)-ergodic.
    \item \(\mathbf{X}\) is \(G\)-weakly mixing \(\iff\) \(\widehat{\mathbf{X}}\) is \(\widehat{G}\)-weakly mixing.
\end{enumerate}
\end{thm}

\begin{rem}
Recall that for a set \(G\), every Gaussian \(G\)-process \(\mathbf{X}\) induces the so called \emph{canonical metric} on \(G\), which is in fact a pseudo-metric, and is given by
\[d_{\mathbf{X}}\left(g,h\right)\coloneq\mathbb{E}\big[\left|X_{g}-X_{h}\right|^{2}\big]^{1/2},\quad g,h\in G.\]
This is well-studied in Gaussian and \(\alpha\)-stable processes; see \cite[\S1.3]{adler2007random}, \cite[\S2.1--2.3]{talagrand2005generic}. For a general \(G\)-process \(\mathbf{X}\), one can naturally work with the version of the canonical metric given by
\[d_{\mathbf{X}}\left(g,h\right)\coloneq\mathbb{E}\big[\left|X_{g}-X_{h}\right|\wedge 1\big],\quad g,h\in G.\]
Note that \(\mathbf{X}\) is separable in probability precisely when \(d_{\mathbf{X}}\) is separable as a pseudo-metric. When \(G\) is a group and \(\mathbf{X}\) is stationary, then \(d_{\mathbf{X}}\) is invariant as a pseudo-metric. In the setup of \cref{thm:stochext}, one could attempt to construct the Polish group \(\widehat{G}\) as a completion of \(G/d_{\mathbf{X}}\) with respect to \(d_\mathbf{X}\) (which forms a genuine metric on the quotient). Our construction of \(\widehat{G}\) is different, and relies on the Polish structure of the group of automorphism of standard probability spaces.
\end{rem}

\subsection{\emph{Lost in translation}: stationary processes and ergodic theory}

The following part provides necessary background for constructing stochastically continuous extensions and proving \cref{thm:stochext}. The results in this section should be regarded as part of the folklore, at least for locally compact second countable groups, and various treatments of related aspects can be found in the literature. We include here proofs in a generality that we could not find in the literature. Due to the particularly wide spectrum of literature, we will only mention here Maruyama's work~\cite[\S3]{maruyama1966transformations}, and some other references will be given later on.

\subsubsection{Separability in probability}\label{sct:sip}

A \(G\)-process \(\mathbf{X}=\left(X_{g}\right)_{g\in G}\) is \textbf{separable in probability}, if there is a countable set \(S\subseteq G\) such that for all \(g\in G\) there is a sequence \(\left(s_{n}\left(g\right)\right)_{n=1}^{\infty}\) in \(S\) with
\[X_{s_{n}\left(g\right)}\pto X_{g}\quad \text{as }n\to\infty.\]
Recall that a measurable space \(\left(\Omega,\mathcal{F}\right)\) is a {\bf standard Borel space}, if \(\mathcal{F}\) is the Borel \(\sigma\)-algebra of some complete separable metric on \(\Omega\). It is a classical fact that every measurable subset of a standard Borel space is standard Borel subspace; see \cite[\S12.B]{kechris2012classical}. A probability space \(\left(\Omega,\mathcal{F},\mathbb{P}\right)\) is called a {\bf standard probability space} if \(\left(\Omega,\mathcal{F}\right)\) is a standard Borel space.\footnote{In some sources, and unlike our terminology, the term \emph{standard probability space} refers to the completion of a standard Borel space with respect to a probability measure, that is, \emph{Lebesgue space} or \emph{Lebesgue--Rokhlin space}.} When \(\left(\Omega,\mathcal{F},\mathbb{P}\right)\) is a standard probability space, \(L^{0}\left(\Omega,\mathcal{F},\mathbb{P}\right)\) is a complete separable metric space with the metric of stochastic convergence (convergence in probability); see \cite[\S17.F]{kechris2012classical}.

\begin{prop}\label{prop:stand}
Let \(G\) be a set and \(\mathbf{X}\) a \(G\)-process. The following properties are equivalent:
\begin{enumerate}
    \item \(\mathbf{X}\) is separable in probability.
    \item There is a \(G\)-process \(\mathbf{X}^{\prime}\) with \(\mathbf{X}^{\prime}\eqd \mathbf{X}\), and \(\mathbf{X}^{\prime}\) is defined on a standard probability space.
\end{enumerate}
Moreover, in this case, the underlying standard probability space can be chosen so that its \(\sigma\)-algebra is generated by countably many elements of \(\mathbf{X}^{\prime}\).
\end{prop}

\begin{rem}
When \(\mathbf{X}=\left(X_{g}\right)_{g\in G}\) is defined on \(\left(\Omega,\mathcal{F},\mathbb{P}\right)\), \(\mathcal{F}\) is considered generated by countably many elements of \(\mathbf{X}\) if there is a countable set \(S\subseteq G\) with \(\mathcal{F}=\sigma\left(X_{s}:s\in S\right)\) modulo \(\mathbb{P}\).
\end{rem}

\begin{proof}[Proof of \cref{prop:stand}]
If either \(\mathbf{X}=\left(X_{g}\right)_{g\in G}\) itself or any \(G\)-process equal in distribution to \(\mathbf{X}\) is defined on a standard probability space \(\left(\Omega,\mathcal{F},\mathbb{P}\right)\), then since \(L^{0}\left(\Omega,\mathcal{F},\mathbb{P}\right)\) is a separable metric space in the metric of stochastic convergence, then its subset \(\{X_{g}:g\in G\}\subseteq L^{0}\left(\Omega,\mathcal{F},\mathbb{P}\right)\) inherits this property, namely \(\mathbf{X}\) is separable in probability. We then prove the other implication. Let \(\mathbf{X}=\left(X_{g}\right)_{g\in G}\) be defined on \(\left(\Omega,\mathcal{F},\mathbb{P}\right)\), and assume it is separable in probability by a countable set \(S\subseteq G\). Since \(S\) is countable, we may pick pointwise defined versions of \(X_{s}\), which we denote again by \(X_{s}:\Omega\to\mathbb{R}\), simultaneously for every \(s\in S\). Define the countably generated \(\sigma\)-algebra \(\mathcal{S}\coloneq\sigma\left(X_{s}:s\in S\right)\subseteq\mathcal{F}\). Fix an arbitrary \(g\in G\), and define an \(\mathcal{S}\)-measurable function \(X_{g}^{\prime}:\Omega\to\mathbb{R}\) with \(\mathbb{P}\left(X_{g}^{\prime}=X_{g}\right)=1\) as follows.\footnote{While \(\{X_{g}^{\prime}=X_{g}\}\) may not belong to \(\mathcal{S}\), this will not be an issue.} Pick a sequence \(\left(s_{n}\left(g\right)\right)_{n=1}^{\infty}\) in \(S\) such that \(X_{s_{n}\left(g\right)}\xrightarrow{\mathrm{a.s}}X_{g}\) as \(n\to\infty\),\footnote{Every sequence convergent in probability admits a subsequence convergent almost surely.} and let
\[\Omega_{g}\coloneq\left\{ \lim\nolimits_{n\to\infty}X_{s_{n}\left(g\right)}\text{ exists}\right\}.\]
Then \(\Omega_{g}\in\mathcal{S}\) and \(\mathbb{P}\left(\Omega_{g}\right)\geq\mathbb{P}\left(\lim\nolimits_{n\to\infty}X_{s_{n}\left(g\right)}=X_{g}\right)=1\). Define then
\[X_{g}^{\prime}:\Omega\to\mathbb{R},\quad X_{g}^{\prime}\left(\omega\right)\coloneq\begin{cases}
\lim\nolimits_{n\to\infty}X_{s_{n}\left(g\right)}\left(\omega\right) & \omega\in\Omega_{g}\\
0 & \omega\notin\Omega_{g}
\end{cases}.\]
Then \(X_{g}^{\prime}\) is \(\mathcal{S}\)-measurable and \(\mathbb{P}\left(X_{g}^{\prime}=X_{g}\right)\geq\mathbb{P}\left(\left\{ \lim\nolimits_{n\to\infty}X_{s_{n}\left(g\right)}=X_{g}\right\} \cap\Omega_{g}\right)=1\). Therefore, \(\mathbf{X}^{\prime}\coloneq\left(X_{g}^{\prime}\right)_{g\in G}\) is defined on \(\left(\Omega,\mathcal{S},\mathbb{P}\right)\), and \(\mathbf{X}'\eqd\mathbf{X}\) since \(\mathbb{P}\left(X_{g}^{\prime}=X_{g}\right)=1\) for every \(g\in G\). In particular, the distribution of \(\mathbf{X}^{\prime}\) does not depend on the choices of the sequences \(\left(s_{n}\right)_{n=1}^{\infty}\). Consider the standard Borel space \(\left(\mathbb{R}^{S},\mathcal{B}^{S}\right)\) and define
\[ \psi:\Omega\to\mathbb{R}^{S},\quad\psi\left(\omega\right)\coloneq\left(X_{s}^{\prime}\left(\omega\right)\right)_{s\in S}. \]
Then not only \(\psi\) is measurable but \(\mathcal{S}=\psi^{-1}\left(\mathcal{B}^{S}\right)\). Let \(\mathbb{P}^{\prime}\coloneq\mathbb{P}\circ\psi^{-1}\), and then \(\left(\mathbb{R}^{S},\mathcal{B}^{S},\mathbb{P}^{\prime}\right)\) is a standard probability space. For every \(g\in G\), since \(X_{g}^{\prime}\) is \(\mathcal{S}=\sigma\left(\psi\right)\)-measurable, by the factorization lemma (see \cite[Lemma~1.14]{kallenberg2021foundations}) there is a Borel function
\[X_{g}^{\prime\prime}:\mathbb{R}^{S}\to\mathbb{R}\text{ satisfying }X_{g}^{\prime}=X_{g}^{\prime\prime}\circ\psi.\]
Then define \(\mathbf{X}^{\prime\prime}\coloneq\left(X_{g}^{\prime\prime}\right)_{g\in G}\) on \(\left(\mathbb{R}^{S},\mathcal{B}^{S},\mathbb{P}^{\prime}\right)\), and one checks that \(\mathbf{X}^{\prime\prime}\eqd\mathbf{X}^{\prime}\eqd\mathbf{X}\). Finally, \(\mathcal{B}^{S}=\sigma\left(X_{s}^{\prime\prime}:s\in S\right)\) modulo \(\mathbb{P}^{\prime}\) since for \(s\in S\), if we let \(\pi_{s}:\mathbb{R}^{S}\to\mathbb{R}\) be the canonical projection to the \(s\)-coordinate, then \(X_{s}^{\prime\prime}\circ\psi=X_{s}^{\prime}=\pi_{s}\circ\psi\) \(\mathbb{P}\)-a.s.
\end{proof}

\subsubsection{Stochastic continuity}

Suppose \(G\) is a Polish group. A \(G\)-process \(\mathbf{X}=\left(X_{g}\right)_{g\in G}\) is called {\bf stochastically continuous} if
\[X_{g}\pto X_{g_{o}}\quad\text{as }g\to g_{o}\text{ for every }g_{o}\in G.\]
Since \(G\) is Polish and hence separable, stochastic continuity implies separability in probability. Note that when \(\mathbf{X}\) is stationary, it suffices to verify stochastic continuity at one arbitrary \(g_{o}\in G\). In the following we focus exclusively on stationary \(G\)-processes with a Polish group \(G\), and characterize their stochastic continuity through the lens of probability preserving actions. We start by considering a general notion of a probability preserving action following Zimmer \cite[\S3]{zimmer1976extensions}:

\begin{defn}\label{dfn:Zimmer}
Let \(G\) be a Polish group. A {\bf probability preserving action} of \(G\) on a standard probability space \(\left(\Omega,\mathcal{F},\mathbb{P}\right)\), is a measurable map \(\mathbf{T}:G\times\Omega\to\Omega\), \(\mathbf{T}:\left(g,\omega\right)\mapsto T_{g}\left(\omega\right)\), such that:
\begin{enumerate}
    \item \(T_{e_{G}}=\mathrm{Id}_{\Omega}\) \(\mathbb{P}\)-a.s.
    \item \(T_{g}\circ T_{h}=T_{gh}\) \(\mathbb{P}\)-a.s. for all \(g,h\in G\).\footnote{Importantly, the set of probability \(1\) on which this holds may depend on \(g,h\).}
    \item \(\mathbb{P}\left(T_{g}^{-1}\left(A\right)\right)=\mathbb{P}\left(A\right)\) for all \(g\in G\) and \(A\in\mathcal{F}\).
\end{enumerate}
We denote such a probability preserving action by \(\mathbf{T}:G\curvearrowright\left(\Omega,\mathcal{F},\mathbb{P}\right)\).
\end{defn}

In contrast to probability preserving \emph{systems}, probability preserving \emph{actions} deal with Polish groups acting on standard probability spaces, and require a jointly measurable map \(G\times\Omega\to\Omega\). We emphasize that both the acting group being Polish and the underlying probability space being standard are integral to the setup of probability preserving actions. The following proposition appears to be well-known, at least for locally compact abelian groups.

\begin{prop}\label{prop:folk}
Let \(G\) be a Polish group and \(\mathbf{X}=\left(X_{g}\right)_{g\in G}\) a stationary \(G\)-process. The following properties are equivalent:
\begin{enumerate}
    \item \(\mathbf{X}\) is stochastically continuous.
    \item There is a probability preserving action \(\mathbf{T}:G\curvearrowright\left(\Omega,\mathcal{F},\mathbb{P}\right)\) and a distinguished random variable \(X_{o}\in L^{0}\left(\Omega,\mathcal{F},\mathbb{P}\right)\), such that \(\mathbf{X}\eqd\left(X_{o}\circ T_{g^{-1}}\right)_{g\in G}\).\footnote{The presentation \(\big(X_{o}\circ T_{g^{-1}}\big)_{g\in G}\) reflects our choice to work with left-stationarity.}
\end{enumerate}
\end{prop}

\begin{rem}
The conditions in \cref{prop:folk} are essentially equivalent to that \(\mathbf{X}\) is \emph{measurable}, i.e. the map \(G\times\Omega\to\mathbb{R}\), \(\left(g,\omega\right)\mapsto X_{g}\left(\omega\right)\), is jointly measurable. It is well-known that measurable stationary \(G\)-processes are stochastically continuous; see \cite[Lemma~6]{peled2004} (cf. \cite[\S1.3]{samorodnitsky2016stochastic}). Note, however, that this requires an a priori map \(\left(g,\omega\right)\mapsto X_{g}\left(\omega\right)\). Alternatively, one defines measurability of \(\mathbf{X}\) by requiring the set \(\{X_{g}:g\in G\}\subseteq L^{0}\left(\Omega,\mathcal{F},\mathbb{P}\right)\) to be measurable. For separable in probability stationary \(G\)-processes it is essentially the same as measurability.\footnote{Using \cref{prop:stand} and with the help of Lebesgue differentiation theorem, as observed in \cite[pp.~307-308]{glasner2005}.}
\end{rem}

\Cref{prop:folk} is the result of two key ingredients. The first is a characterization of stochastic continuity formulated by Pursley~\cite{pursley1977}, which is closely related to a classical theorem of Hoffmann-J\o rgensen~\cite{hoffmann1973existence} (see also~\cite[\S9.4]{samorodnitsky1994stable}). While these theorems were formulated for \(G=\mathbb{R}\), they can be generalized to Polish groups. The second ingredient is an observation of Glasner, Tsirelson, and Weiss~\cite[pp.~307--308]{glasner2005}, which relies on the automatic continuity of measurable homomorphisms between Polish groups, and the correspondence between probability preserving automorphisms and homomorphisms of measure algebras.

\begin{thm}[Following Hoffmann-J\o rgensen/Pursley]\label{thm:Pursley}
Let \(G\) be a Polish group. A stationary \(G\)-process \(\mathbf{X}\) is stochastically continuous if and only if the following functions are all continuous:\footnote{In stationary processes, this property was called \emph{continuity in the sense of Pinsker} in \cite{pursley1977}. In ergodic theory, it is the weak continuity of the Koopman representation.}
\[G\to\left[0,1\right],\quad g\mapsto\mathbb{P}_{\mathbf{X}}\left(R_{g}^{-1}\left(E\right)\triangle E\right),\quad E\in\mathcal{B}^{G}.\]
\end{thm}

We include in Appendix~\ref{app:Pursley} an elementary proof of \cref{thm:Pursley} essentially due to Pursley. Next, we recall an alternative description of probability preserving actions as put by Glasner, Tsirelson, and Weiss \cite[pp.~307-308]{glasner2005}. For a standard probability space \(\left(\Omega,\mathcal{F},\mathbb{P}\right)\) (and only when it is standard), define the {\bf automorphism group},
\[\mathrm{Aut}\left(\Omega,\mathcal{F},\mathbb{P}\right),\]
to be the group of all equivalence classes of probability preserving automorphisms of \(\left(\Omega,\mathcal{F},\mathbb{P}\right)\), where two such automorphisms are considered to be equivalent if they coincide \(\mathbb{P}\)-a.s. An element of \(\mathrm{Aut}\left(\Omega,\mathcal{F},\mathbb{P}\right)\) is called an {\bf automorphism} of \(\left(\Omega,\mathcal{F},\mathbb{P}\right)\). The group \(\mathrm{Aut}\left(\Omega,\mathcal{F},\mathbb{P}\right)\) becomes a Polish group in the weak topology, defined to be generated by the functions
\[\mathrm{Aut}\left(\Omega,\mathcal{F},\mathbb{P}\right)\to\left[0,1\right],\quad T\mapsto\mathbb{P}\left(T^{-1}\left(A\right)\triangle A\right),\quad A\in\mathcal{F}.\]
Every probability preserving action \(\mathbf{T}:G\curvearrowright\left(\Omega,\mathcal{F},\mathbb{P}\right)\) naturally induces a homomorphism \(\mathbf{T}:G\to\mathrm{Aut}\left(\Omega,\mathcal{F},\mathbb{P}\right)\), which is measurable thanks to the joint measurability of the action, and it is then continuous thanks to the automatic continuity property \cite[\S9.C]{kechris2012classical}. It was observed in \cite[pp.~307-308]{glasner2005}, using the Lebesgue differentiation theorem, that the converse is also true, and thus we get the following:

\begin{fct}[Glasner--Tsirelson--Weiss' observation]\label{fct:GTW}
For every Polish group \(G\) and a standard probability space \(\left(\Omega,\mathcal{F},\mathbb{P}\right)\), giving a probability preserving action \(\mathbf{T}:G\curvearrowright\left(\Omega,\mathcal{F},\mathbb{P}\right)\) (as in \cref{dfn:Zimmer}) is the same as giving a continuous homomorphism \(\mathbf{T}:G\to\mathrm{Aut}\left(\Omega,\mathcal{F},\mathbb{P}\right)\).
\end{fct}

Recall that the \emph{Boolean measure algebra} \(\mathrm{MAlg}\left(\mathcal{F},\mathbb{P}\right)\) consists of equivalence classes (modulo \(\mathbb{P}\)) of elements of \(\mathcal{F}\), which is a complete separable metric space in the metric \(d_{\mathbb{P}}\left(A,A^{\prime}\right)=\mathbb{P}\left(A\triangle A^{\prime}\right)\). In this context we may write \(A\eqp A^{\prime}\) for the occurrence of \(\mathbb{P}\left(A\triangle A^{\prime}\right)=0\). The following is a classical fact; see \cite[\S15.C,\,\S17.F]{kechris2012classical}.

\begin{fct}\label{fct:bool}
Let \(\left(\Omega,\mathcal{F},\mathbb{P}\right)\) be a standard probability space and \(\Phi:\mathrm{MAlg}\left(\mathcal{F},\mathbb{P}\right)\to\mathrm{MAlg}\left(\mathcal{F},\mathbb{P}\right)\) a Boolean isometry.\footnote{A \emph{Boolean} map takes \(0\) to \(0\) (the representative of \(\emptyset\)), and respects complements and countable unions. Note that if an isometry preserves finite unions, then by continuity it preserves also countable union.} Then there is a unique automorphism \(T\in\mathrm{Aut}\left(\Omega,\mathcal{F},\mathbb{P}\right)\) such that
\[\Phi\left(A\right)\eqp T^{-1}\left(A\right)\text{ for every }A\in\mathrm{MAlg}\left(\mathcal{F},\mathbb{P}\right).\]
\end{fct}

\begin{lem}\label{lem:tau}
Let \(G\) be a group and \(\mathbf{X}=\left(X_{g}\right)_{g\in G}\) a stationary \(G\)-process defined on a standard probability space \(\left(\Omega,\mathcal{F},\mathbb{P}\right)\), such that \(\mathcal{F}\) is generated by countably many elements of \(\mathbf{X}\). Then there is a unique homomorphism \(\tau:G\to\mathrm{Aut}\left(\Omega,\mathcal{F},\mathbb{P}\right)\) specified by the property
\[\mathbb{P}\left(X_{h}\circ\tau\left(g\right)=X_{g^{-1}h}\right)=1\text{ for every }g,h\in G.\]
\end{lem}

\begin{proof}[Proof of \cref{lem:tau}]
Fix a countable set \(S\subseteq G\) with \(\mathcal{F}=\sigma\left(X_{s}:s\in S\right)\). Define an algebra \(\mathcal{A}\subset\mathcal{F}\) to consist of finite unions of the sets
\[\left\{X_{s_{1}}\in E_{1},\dotsc,X_{s_{n}}\in E_{n}\right\},\text{ for }s_{1},\dotsc,s_{n}\in S\text{ and }E_{1},\dotsc,E_{n}\in\mathcal{B}.\]
By the assumption \(\mathcal{F}=\sigma\left(\mathcal{A}\right)\), hence the algebra \(\mathcal{A}/\mathbb{P}\) is dense in \(\mathrm{MAlg}\left(\mathcal{F},\mathbb{P}\right)\) and generates it as a Boolean measure algebra. Given \(g\in G\), define \(\Phi\left(g\right):\mathcal{A}/\mathbb{P}\to\mathrm{MAlg}\left(\mathcal{F},\mathbb{P}\right)\) by
\[\Phi\left(g\right):\left\{X_{s_{1}}\in E_{1},\dotsc,X_{s_{n}}\in E_{n}\right\} \mapsto\left\{X_{g^{-1}s_{1}}\in E_{1},\dotsc,X_{g^{-1}s_{n}}\in E_{n}\right\}.\]
One verifies that \(\Phi\left(g\right)\) takes \(0\) to \(0\) and respects complements and finite unions, and that by the stationarity it is an isometry on \(\mathcal{A}/\mathbb{P}\). Then by the continuity it uniquely extends to an isometric Boolean embedding \(\Phi\left(g\right):\mathrm{MAlg}\left(\mathcal{F},\mathbb{P}\right)\to\mathrm{MAlg}\left(\mathcal{F},\mathbb{P}\right)\). Moreover, \(\Phi\left(g\right)\circ\Phi\left(g^{-1}\right)=\mathrm{Id}\) and \(\Phi\left(g^{-1}\right)\circ\Phi\left(g\right)=\mathrm{Id}\) on \(\mathcal{A}/\mathbb{P}\), hence also on \(\mathrm{MAlg}\left(\mathcal{F},\mathbb{P}\right)\), so \(\Phi\left(g\right)\) is a Boolean isometry of \(\mathrm{MAlg}\left(\mathcal{F},\mathbb{P}\right)\). Using \cref{fct:bool}, there is a unique automorphism \(\tau\left(g\right)\in\mathrm{Aut}\left(\Omega,\mathcal{F},\mathbb{P}\right)\) satisfying \(\Phi\left(g\right)\left(\cdot\right)=\tau\left(g\right)^{-1}\left(\cdot\right)\). Let us verify that the desired homomorphism is formed by
\[\tau:G\to\mathrm{Aut}\left(\Omega,\mathcal{F},\mathbb{P}\right),\quad g\mapsto\tau\left(g\right).\]
First, since \(\Phi\left(e_{G}\right)=\mathrm{Id}_{\mathrm{MAlg}\left(\mathcal{F},\mathbb{P}\right)}\) and \(\Phi\left(g_{1}\right)\circ\Phi\left(g_{2}\right)=\Phi\left(g_{2}g_{1}\right)\) for all \(g_{1},g_{2}\in G\) (suffices to be verified on \(\mathcal{A}/\mathbb{P}\)), \(\tau\) is a homomorphism. We then must verify that \(\mathbb{P}\left(X_{h}\circ\tau\left(g\right)=X_{g^{-1}h}\right)=1\) for all \(g,h\in G\). Suppose first \(h\in S\). Then for every \(g\in G\) and an arbitrary \(E\in\mathcal{B}\),
\begin{align*}
\mathbb{P}\left(\left\{X_{h}\circ\tau\left(g\right)\in E\right\} \triangle\left\{X_{g^{-1}h}\in E\right\}\right)
&=\mathbb{P}\big(\tau\left(g\right)^{-1}\left(\left\{ X_{h}\in E\right\}\right)\triangle\left\{X_{g^{-1}h}\in E\right\}\big)\\
&=\mathbb{P}\left(\Phi\left(g\right)\left(\left\{X_{h}\in E\right\} \right)\triangle\left\{X_{g^{-1}h}\in E\right\}\right)=0,
\end{align*}
hence \(\mathbb{P}\left(X_{h}\circ\tau\left(g\right)=X_{g^{-1}h}\right)=1\). Now for \(h\in G\) which might not be in \(S\), for each \(g\in G\) repeat the above construction of \(\tau\left(g\right)\) with \(S\) replaced by \(S\cup\{h\}\). Since \(\mathcal{F}=\sigma\left(X_{s}:s\in S\cup\{h\}\right)\), the uniqueness of each step in the construction yields the same automorphism \(\tau\left(g\right)\). Then the foregoing argument, this time for \(h\in S\cup\{h\}\), applies, hence \(\mathbb{P}\left(X_{h}\circ\tau\left(g\right)=X_{g^{-1}h}\right)=1\).
\end{proof}

\begin{proof}[Proof of \cref{prop:folk}]
\((2)\implies(1)\). Let \(\mathbf{T}:G\curvearrowright\left(\Omega,\mathcal{F},\mathbb{P}\right)\) be a probability preserving action, \(X_{o}\in L^{0}\left(\Omega,\mathcal{F},\mathbb{P}\right)\), and look at \(\mathbf{X}_{o}=\left(X_{g}\right)_{g\in G}\coloneq\left(X_{o}\circ T_{g^{-1}}\right)_{g\in G}\).\footnote{Note that for the purpose of this definition, we may assume that each \(T_{g}\) is defined pointwise, even though the relations \(T_{g}\circ T_{h}=T_{gh}\) hold only almost everywhere.} For every \(E\in\mathcal{B}^{G}\) denote
\[\overline{E}\coloneq\left\{ \omega\in\Omega:\left(X_{h}\left(\omega\right)\right)_{h\in G}\in E\right\}\in\mathcal{F}.\]
Note that \(\mathbb{P}\left(\overline{E}\right)=\mathbb{P}_{\mathbf{X}_{o}}\left(E\right)\) for every \(E\in\mathcal{B}^{G}\), and that the map \(E\mapsto\overline{E}\) is Boolean, hence
\begin{equation}\label{eq:CBid}
\mathbb{P}_{\mathbf{X}_{o}}\left(R_{g}^{-1}\left(E\right)\triangle E\right)=\mathbb{P}\big(\overline{R_{g}^{-1}\left(E\right)\triangle E}\big)=\mathbb{P}\big(\overline{R_{g}^{-1}\left(E\right)}\triangle\overline{E}\big),\quad E\in\mathcal{B}^{G}.
\end{equation}
We now claim that for every \(E\in\mathcal{B}^{G}\) and \(g\in G\) the following relation holds:
\begin{align*}
T_{g}^{-1}\big(\overline{E}\big)
&=\left\{ \omega\in\Omega:\left(X_{o}\left(T_{h^{-1}}\left(T_{g}\left(\omega\right)\right)\right)\right)_{h\in G}\in E\right\}\\
(\ast)&\eqp\left\{ \omega\in\Omega:\left(X_{g^{-1}h}\left(\omega\right)\right)_{h\in G}\in E\right\}\\
&=\left\{ \omega\in\Omega:\left(X_{h}\left(\omega\right)\right)_{h\in G}\in R_{g}^{-1}\left(E\right)\right\}=\overline{R_{g}^{-1}\left(E\right)},
\end{align*}
where \((\ast)\) is because by virtue of \(E\in\mathcal{B}^{G}\), there is some countable set \(S\) in \(G\) such that \(E\in\mathcal{B}^{S}\), hence \(\mathbb{P}\left(\bigcap_{s\in S}\left\{ X_{o}\circ T_{s^{-1}}\circ T_{g}=X_{o}\circ T_{s^{-1}g}=X_{g^{-1}s}\right\} \right)=1\), which readily implies \((\ast)\). Then together with \eqref{eq:CBid} we conclude
\[\mathbb{P}_{\mathbf{X}_{o}}\left(R_{g}^{-1}\left(E\right)\triangle E\right)=\mathbb{P}\big(T_{g}^{-1}\big(\overline{E}\big)\triangle\overline{E}\big).\]
Now by Glasner--Tsirelson--Weiss' observation (\cref{fct:GTW}), the map \(G\to\mathrm{Aut}\left(\Omega,\mathcal{F},\mathbb{P}\right)\), \(g\mapsto T_{g}\), is continuous, and by the definition of the topology of \(\mathrm{Aut}\left(\Omega,\mathcal{F},\mathbb{P}\right)\) this means
\[\mathbb{P}_{\mathbf{X}_{o}}\left(R_{g}^{-1}\left(E\right)\triangle E\right)=\mathbb{P}\big(T_{g}^{-1}\big(\overline{E}\big)\triangle\overline{E}\big)\to 0\quad\text{as }g\to e_{G}.\]
Therefore, from \cref{thm:Pursley} it follows that \(\mathbf{X}\) is stochastically continuous.

\smallskip

\((1)\implies(2)\). Suppose \(\mathbf{X}=\left(X_{g}\right)_{g\in G}\) is stochastically continuous. By \cref{prop:stand}, after replacing \(\mathbf{X}\) with a stationary \(G\)-process equal in distribution, we may assume that \(\mathbf{X}\) is defined on a standard probability space \(\left(\Omega,\mathcal{F},\mathbb{P}\right)\) and \(\mathcal{F}\) is generated by countably many elements of \(\mathbf{X}\). Then \cref{lem:tau} applies, with the additional information that \(G\) is Polish and \(\mathbf{X}\) is stochastically continuous. Consider the homomorphism \(\tau:G\to\mathrm{Aut}\left(\Omega,\mathcal{F},\mathbb{P}\right)\) from \cref{lem:tau}, and we show it is continuous. Fix a countable set \(S\subseteq G\) with \(\mathcal{F}=\sigma\left(X_{s}:s\in S\right)\). Let \(\mathcal{D}\subseteq\mathcal{B}\) be a countable algebra generating \(\mathcal{B}\), and let \(\mathcal{A}\subset\mathcal{F}\) be the algebra of finite unions of sets
\[\left\{X_{s_{1}}\in D_{1},\dotsc,X_{s_{n}}\in D_{n}\right\}\quad\text{ for }s_{1},\dotsc,s_{n}\in S\text{ and }D_{1},\dotsc,D_{n}\in\mathcal{D}.\]
Then \(\sigma\left(\mathcal{A}\right)=\mathcal{F}\) and \(\mathcal{A}/\mathbb{P}\) is dense in \(\mathrm{MAlg}\left(\mathcal{F},\mathbb{P}\right)\). For \(A=\left\{X_{s_{1}}\in D_{1},\dotsc,X_{s_{n}}\in D_{n}\right\}\in\mathcal{A}\) let
\[\overline{A}\coloneq\left\{\left(x_{h}\right)_{h\in G}:x_{s_{1}}\in D_{1},\dotsc,x_{s_{n}}\in D_{n}\right\}\in\mathcal{B}^{G}.\]
Then for every \(A\in\mathcal{A}\) and \(g\in G\), by the defining property of \(\tau\) as in \cref{lem:tau} we have
\[\mathbb{P}\big(\tau\left(g\right)^{-1}\left(A\right)\triangle A\big)=\mathbb{P}_{\mathbf{X}}\big(R_{g}^{-1}\big(\overline{A}\big)\triangle\overline{A}\big)=\mathbb{P}_{\mathbf{X}}\big(R_{g^{-1}}\big(\overline{A}\big)\triangle\overline{A}\big).\]
Then for every \(A\in\mathcal{A}\), by the stochastic continuity of \(\mathbf{X}\) and \cref{thm:Pursley}, the function \(G\to\left[0,1\right]\), \(g\mapsto\mathbb{P}\big(\tau\left(g\right)^{-1}\left(A\right)\triangle A\big)\), is continuous. Since \(\mathcal{A}/\mathbb{P}\) is dense in \(\mathrm{MAlg}\left(\mathcal{F},\mathbb{P}\right)\), the family of functions
\[\mathrm{Aut}\left(\Omega,\mathcal{F},\mathbb{P}\right)\to\left[0,1\right],\quad T\mapsto\mathbb{P}\left(T^{-1}\left(A\right)\triangle A\right),\quad A\in\mathcal{A},\]
generates the topology of \(\mathrm{Aut}\left(\Omega,\mathcal{F},\mathbb{P}\right)\). Hence \(\tau:G\to\mathrm{Aut}\left(\Omega,\mathcal{F},\mathbb{P}\right)\) is continuous.

\smallskip

In light of Glasner--Tsirelson--Weiss' observation (\cref{fct:GTW}), \(\tau\) corresponds to a probability preserving action \(\mathbf{T}:G\curvearrowright\left(\Omega,\mathcal{F},\mathbb{P}\right)\). Denote \(T_{g}=\tau\left(g\right)\), \(g\in G\), for consistency. Put the random variable \(X_{o}\coloneq X_{e_{G}}\in L^{0}\left(\Omega,\mathcal{F},\mathbb{P}\right)\). By the defining property of \(\tau\) in \cref{lem:tau} (with \(h=e_{G}\)),
\[\mathbb{P}\left(X_{o}\circ T_{g^{-1}}=X_{g}\right)=\mathbb{P}\left(X_{e_{G}}\circ\tau\left(g^{-1}\right)=X_{g}\right)=1\text{ for every }g\in G.\qedhere\]
\end{proof}

\subsection{Proof of \cref{thm:stochext}}

Let \(\mathbf{X}=\left(X_{g}\right)_{g\in G}\) be as in the theorem. Since it is separable in probability, then by \cref{prop:stand}, after replacing \(\mathbf{X}\) with a stationary \(G\)-process equal in distribution, we may assume that \(\mathbf{X}\) is defined on a standard probability space \(\left(\Omega,\mathcal{F},\mathbb{P}\right)\) and \(\mathcal{F}\) is generated by countably many elements of \(\mathbf{X}\). Consider the homomorphism \(\tau:G\to\mathrm{Aut}\left(\Omega,\mathcal{F},\mathbb{P}\right)\) obtained in \cref{lem:tau}, and look at the closure of the subgroup \(\tau\left(G\right)\),
\[\widehat{G}\coloneq\mathrm{cl}_{\mathrm{Aut}\left(\Omega,\mathcal{F},\mathbb{P}\right)}\left(\tau\left(G\right)\right)\leq\mathrm{Aut}\left(\Omega,\mathcal{F},\mathbb{P}\right).\]
Since a closure of a subgroup is a closed subgroup, and since a closed subgroup of a Polish group is a Polish group in the subspace topology, \(\widehat{G}\) is a Polish group. We now define
\[\widehat{\mathbf{X}}=\big(\widehat{X}_{T}\big)_{T\in\widehat{G}},\quad\widehat{X}_{T}\coloneq X_{e_{G}}\circ T^{-1},\quad T\in\widehat{G}.\]
(note \(\tau\left(e_{G}\right)=e_{\mathrm{Aut}\left(\Omega,\mathcal{F},\mathbb{P}\right)}=e_{\widehat{G}}\)). Since \(\widehat{G}\) is closed in \(\mathrm{Aut}\left(\Omega,\mathcal{F},\mathbb{P}\right)\), by the Glasner--Tsirelson--Weiss' observation (\cref{fct:GTW}) together with \cref{prop:folk}, \(\widehat{\mathbf{X}}\) is stochastically continuous. It is clear that \(\tau\left(G\right)\) is dense in \(\widehat{G}\). Additionally, by the definition of \(\tau\) as in \cref{lem:tau},
\[\mathbb{P}\big(\widehat{X}_{\tau\left(g\right)}=X_{g}\big)=\mathbb{P}\big(X_{e_{G}}\circ\tau\left(g\right)^{-1}=X_{g}\big)=1\text{ for every }g\in G.\]
This establishes the existence of a stochastically continuous extension. Before proving the last part of the theorem, we formulate a general useful lemma.

\begin{lem}\label{lem:ergdens}
Let \(\Gamma\) be a Polish group, \(G\leq\Gamma\) a dense subgroup, and \(\mathbf{X}=\left(X_{\gamma}\right)_{\gamma\in\Gamma}\) a stochastically continuous stationary \(\Gamma\)-process. Consider the stationary \(G\)-process \(\mathbf{X}\mid_{G}=\left(X_{g}\right)_{g\in G}\). Then:
\begin{enumerate}
    \item \(\mathbf{X}\) is \(\Gamma\)-ergodic \(\iff\) \(\mathbf{X}\mid_{G}\) is \(G\)-ergodic.
    \item \(\mathbf{X}\) is \(\Gamma\)-weakly mixing \(\iff\) \(\mathbf{X}\mid_{G}\) is \(G\)-weakly mixing.
\end{enumerate}
\end{lem}

\begin{proof}[Proof of \cref{lem:ergdens}]
Part (2) follows from part (1) by applying it to the diagonal product system, so we prove part (1). Let \(\mathcal{B}^{\Gamma}\) be the product \(\sigma\)-algebra on \(\mathbb{R}^{\Gamma}\), and consider the \(\sigma\)-algebra \(\mathcal{G}\subseteq\mathcal{B}^{\Gamma}\) on \(\mathbb{R}^{\Gamma}\) generated by the coordinates of \(G\). We claim that the \(\mathbb{P}_{\mathbf{X}}\)-completion of \(\mathcal{G}\) coincides with the \(\mathbb{P}_{\mathbf{X}}\)-completion of \(\mathcal{B}^{\Gamma}\). Indeed, for every \(\gamma\in\Gamma\), \(X_{\gamma}\) is a limit in \(\mathbb{P}_{\mathbf{X}}\)-probability of a sequence of elements from \(\{X_{g}:g\in G\}\), and by passing to a subsequence it is a limit \(\mathbb{P}_{\mathbf{X}}\)-a.s. of a sequence of elements from \(\{X_{g}:g\in G\}\). Since each \(\pi_{g}\) is \(\mathcal{G}\)-measurable, \(\pi_{\gamma}\) is measurable with respect to the \(\mathbb{P}_{\mathbf{X}}\)-completion of \(\mathcal{G}\). Since this holds for every \(\gamma\in\Gamma\), all coordinate maps on \(\mathbb{R}^{\Gamma}\) are measurable with respect to the \(\mathbb{P}_{\mathbf{X}}\)-completion of \(\mathcal{G}\), and hence the two completed \(\sigma\)-algebras coincide. Recall that ergodicity is unchanged under taking measure-completion (as mentioned in \cref{sct:ergwm}), hence \(\mathbf{X}\) is \(\Gamma\)-ergodic if and only if the probability preserving \(\Gamma\)-system \(\mathcal{R}=\{R_{\gamma}\}_{\gamma\in\Gamma}\) on \(\left(\mathbb{R}^{\Gamma},\mathcal{G},\mathbb{P}_{\mathbf{X}}\right)\) is ergodic.

\smallskip

We show that this is equivalent to \(\mathbf{X}\mid_{G}\) being \(G\)-ergodic, namely the probability preserving \(G\)-system \(\mathcal{S}=\{S_{g}\}_{g\in G}\) on \(\left(\mathbb{R}^{G},\mathcal{B}^{G},\mathbb{P}_{\mathbf{X}\mid_{G}}\right)\) is ergodic. Look at the projection map
\[\pi:\left(\mathbb{R}^{\Gamma},\mathcal{G},\mathbb{P}_{\mathbf{X}}\right)\to\left(\mathbb{R}^{G},\mathcal{B}^{G},\mathbb{P}_{\mathbf{X}\mid_{G}}\right),\quad\pi:\left(x_{\gamma}\right)_{\gamma\in\Gamma}\mapsto\left(x_{\gamma}\right)_{\gamma\in G}.\]
Then \(\pi^{-1}\left(\mathcal{B}^{G}\right)=\mathcal{G}\), \(\mathbb{P}_{\mathbf{X}}\circ\pi^{-1}=\mathbb{P}_{\mathbf{X}\mid_{G}}\), and \(\pi\circ R_{g}=S_{g}\circ\pi\) for every \(g\in G\). In particular,
\[\mathbb{P}_{\mathbf{X}\mid_{G}}\left(S_{g}^{-1}\left(E\right)\triangle E\right)
=\mathbb{P}_{\mathbf{X}}\left(R_{g}^{-1}\left(\pi^{-1}\left(E\right)\right)\triangle\pi^{-1}\left(E\right)\right),
\quad E\in\mathcal{B}^{G},\,g\in G.\]
Fix \(E\in\mathcal{B}^{G}\) and put \(A\coloneq\pi^{-1}\left(E\right)\in\mathcal{G}\). By stochastic continuity and \cref{thm:Pursley}, the function
\[\Gamma\to\left[0,1\right],\quad \gamma\mapsto\mathbb{P}_{\mathbf{X}}\left(R_{\gamma}^{-1}\left(A\right)\triangle A\right),\]
is continuous. Since \(G\) is dense in \(\Gamma\), it follows that
\[\forall g\in G,\,\,\mathbb{P}_{\mathbf{X}}\left(R_{g}^{-1}\left(A\right)\triangle A\right)=0\iff\forall\gamma\in\Gamma,\,\,\mathbb{P}_{\mathbf{X}}\left(R_{\gamma}^{-1}\left(A\right)\triangle A\right)=0.\]

\smallskip

Combining this with the previous identity, we obtain
\[\forall g\in G,\,\,\mathbb{P}_{\mathbf{X}\mid_{G}}\left(S_{g}^{-1}\left(E\right)\triangle E\right)=0\iff\forall\gamma\in\Gamma,\,\,\mathbb{P}_{\mathbf{X}}\left(R_{\gamma}^{-1}\left(\pi^{-1}\left(E\right)\right)\triangle\pi^{-1}\left(E\right)\right)=0.\]

Since \(\pi^{-1}\left(\mathcal{B}^{G}\right)=\mathcal{G}\), it follows that the \(\mathcal{S}\)-invariant sets in \(\mathcal{B}^{G}\) correspond exactly (via \(\pi^{-1}\)) to the \(\mathcal{R}\)-invariant sets in \(\mathcal{G}\), hence \(\mathcal{R}\) is ergodic if and only if \(\mathcal{S}\) is ergodic.
\end{proof}

We now deduce properties (i) and (ii) in the theorem. For (i), we argue that
\[\begin{aligned}
\mathbf{X}\text{ is }G\text{-ergodic}
&\quad\Longleftrightarrow\quad\widehat{\mathbf{X}}\mid_{\tau\left(G\right)}\text{ is }G\text{-ergodic}\\
&\quad\Longleftrightarrow\quad\widehat{\mathbf{X}}\mid_{\tau\left(G\right)}\text{ is }\tau\left(G\right)\text{-ergodic}\quad\Longleftrightarrow\quad \widehat{\mathbf{X}}\text{ is }\widehat{G}\text{-ergodic}.
\end{aligned}\]
The first equivalence is because the processes are equal in distribution; in the second equivalence we used two ways to view \(\widehat{\mathbf{X}}\mid_{\tau\left(G\right)}\), as a \(G\)-process (translations given by \(R_{g}:\left(x_{\tau\left(h\right)}\right)_{h\in G}\mapsto\left(x_{\tau\left(g^{-1}h\right)}\right)_{h\in G}\)), or as a \(\tau\left(G\right)\)-process (translations given by \(R_{\tau\left(g\right)}\left(x_{\tau\left(h\right)}\right)_{h\in G}\mapsto\big(x_{\tau\left(g\right)^{-1}\tau\left(h\right)}\big)_{h\in G}\)), and both ways categorically coincide (\(\tau\) is a homomorphism); the third equivalence is by \cref{lem:ergdens} with \(\Gamma=\widehat{G}\). For (ii), observe that if \(\widehat{\mathbf{X}}\) is a stochastically continuous extension of \(\mathbf{X}\), then \(\widehat{\mathbf{X}}\otimes\widehat{\mathbf{X}}\) is a stochastically continuous extension of \(\mathbf{X}\otimes\mathbf{X}\). Then by virtue of (i),
\[\begin{aligned}
\mathbf{X}\text{ is }G\text{-weakly mixing}
&\quad\Longleftrightarrow\quad\mathbf{X}\otimes\mathbf{X}\text{ is }G\text{-ergodic} \\
&\quad\Longleftrightarrow\quad \widehat{\mathbf{X}}\otimes\widehat{\mathbf{X}}\text{ is }\widehat{G}\text{-ergodic}\quad\Longleftrightarrow\quad \widehat{\mathbf{X}}\text{ is }\widehat{G}\text{-weakly mixing}.
\end{aligned}\]

The proof of \cref{thm:stochext} is now complete.

\section{Proof of the main theorem}\label{sct:finalproof}

\subsection{Necessary background}\label{sct:necback}

We give some background about infinitely divisible Poissonian stationary processes, with the goal of having Maruyama representation for such processes with countable indexing group. This type of representation exists in the literature also for more general groups, however we will only need the countable case. We refer the interested reader to the works of Kabluchko and Stoev \cite[\S2.2]{kabluchko2016stochastic} and of Rosinski \cite[\S2]{rosinki2018representations}.

\subsubsection{Infinitely divisible Poissonian stationary processes}\label{sct:idp}

A random vector \(V\in\mathbb{R}^{d}\) is called {\bf infinitely divisible}, if for every \(n\in\mathbb{N}\) there are independent identically distributed random vectors \(V_{1}^{\left(n\right)},\dotsc,V_{n}^{\left(n\right)}\) with \(V\eqd V_{1}^{\left(n\right)}+\dotsm+V_{n}^{\left(n\right)}\). As part of the classical L\'{e}vy--Khintchine representation \cite[\S8]{sato1999levy}, \cite[\S3]{samorodnitsky2016stochastic}, every infinitely divisible random vector \(V\) is identically distributed as \(Z\oplus P\) (independent sum), where \(Z\) is a Gaussian random vector and \(P\) is an \emph{infinitely divisible Poissonian} (henceforth {\sc id}p) random vector. This decomposition is unique up to a constant drift, and therefore it is customary to choose the Gaussian random variable \(Z\) to be centered. The L\'{e}vy--Khintchine representation further asserts that the distribution of \(P\) is determined by a {\bf L\'{e}vy measure} and a {\bf constant drift}, and we will get back to that later on.

\smallskip

Let \(G\) be a group. A stationary \(G\)-process \(\mathbf{X}=\left(X_{g}\right)_{g\in G}\) is called {\bf infinitely divisible} if all its finite dimensional distributions are infinitely divisible as random vectors. {\bf Gaussian \(G\)-processes} and {\bf {\sc id}p \(G\)-processes} are defined analogously. Using the uniqueness in the L\'{e}vy–Khintchine decomposition, every infinitely divisible stationary \(G\)-process \(\mathbf{X}\) is identically distributed as \(\mathbf{Z}\oplus\mathbf{P}\) (independent sum), where \(\mathbf{Z}\) is a centered Gaussian stationary \(G\)-process and \(\mathbf{P}\) is an {\sc id}p stationary \(G\)-process (cf. \cite[\S2.2]{rosinki2018representations}).

\subsubsection{Poisson random measures and stochastic integral}\label{sct:Poissint}

Let \(G\) be a countable set. For any Borel \(\sigma\)-finite measure \(\ell\) on \(\mathbb{R}^{G}\), we denote by \(\mathcal{N}_{\ell}\) the Poisson random measure on \(\mathbb{R}^{G}\) with intensity \(\ell\). Thus, \(\mathcal{N}_{\ell}\) is a random measure on \(\mathbb{R}^{G}\) whose distribution is characterized by the following two properties: first, for every Borel set \(A\) in \(\mathbb{R}^{G}\), the random variable \(\mathcal{N}_{\ell}\left(A\right)\) has Poisson distribution with mean \(\ell\left(A\right)\); second, whenever \(A_{1},\dotsc,A_{n}\) are pairwise disjoint Borel sets in \(\mathbb{R}^{G}\), the random variables \(\mathcal{N}_{\ell}\left(A_{1}\right),\dotsc,\mathcal{N}_{\ell}\left(A_{n}\right)\) are independent (see \cite[Prop.~19.4]{sato1999levy}).

\smallskip

The Poisson random measure \(\mathcal{N}_{\ell}\) can be realized as a random outcome of the standard probability space \(\big(\left(\mathbb{R}^{G}\right)^{\ast},\left(\mathcal{B}^{G}\right)^{\ast},\ell^{\ast}\big)\), where \(\left(\mathbb{R}^{G}\right)^{\ast}\) is a standard Borel space consisting of counting measures on \(\mathbb{R}^{G}\), the \(\sigma\)-algebra \(\left(\mathcal{B}^{G}\right)^{\ast}\) is generated by the evaluation maps \(N_{A}:\theta\mapsto\theta\left(A\right)\), \(A\in\mathcal{B}^{G}\), and the probability measure \(\ell^{\ast}\) is determined by that the family \(\{N_{A}:A\in\mathcal{B}^{G}\}\) forms a Poisson point process; for technical details see \cite[\S3]{avraham2025poissonian} and the references therein.

\smallskip

For a measurable function \(f:\mathbb{R}^{G}\to\mathbb{R}\) satisfying \(\int_{\mathbb{R}^{G}}\left|f\right|^{2}\wedge 1 d\ell<+\infty\) (where \(x\wedge y\) denotes the minimum of two nonnegative numbers \(x\) and \(y\)), define the {\bf Poisson stochastic integral} of \(f\) to be the limit in probability
\begin{equation}\label{eq:plim}
I_{\ell}\left(f\right)\coloneq\mathrm{plim}_{\epsilon\searrow 0}\Big(\int_{\mathbb{R}^{G}}f\cdot\mathbf{1}_{\left\{\left| f\right|>\epsilon\right\}}d\mathcal{N}_{\ell}-\int_{\mathbb{R}^{G}}f\cdot\mathbf{1}_{\left\{\epsilon\leq\left|f\right|\leq 1\right\} }d\ell\Big).
\end{equation}
Then \(I_{\ell}\left(f\right)\) is a well defined {\sc id}p random variable, whose distribution is specified by
\[\log\mathbb{E}\left(\exp\left(itI_{\ell}\left(f\right) \right)\right)=\int_{\mathbb{R}^{G}}\big(e^{itf\left(\mathbf{x}\right)}-1-itf\left(\mathbf{x}\right)\cdot\mathbf{1}_{\{\left|f\left(\mathbf{x}\right)\right|\leq 1\}}\big)d\ell\left(\mathbf{x}\right),\quad t\in\mathbb{R}.\]
For more on the definition of the Poisson stochastic integral see \cite[\S1]{maruyama1970infinitely}, \cite[\S8]{sato1999levy}, \cite[\S2.2]{rosinki2018representations}.

\subsubsection{The Maruyama representation}\label{sct:Maruyama}

Assume now that \(G\) is a countable group, and denote by \(\mathcal{R}=\{R_{g}\}_{g\in G}\) the action by (left) translations of \(G\) on \(\mathbb{R}^{G}\). For \(g\in G\), let the canonical projection
\[\Xi_{g}:\mathbb{R}^{G}\to\mathbb{R},\quad\Xi_{g}:\left(x_{h}\right)_{h\in G}\mapsto x_{g}.\]
Then one has the following recipe to produce {\sc id}p stationary \(G\)-processes. Pick a translation invariant {\bf L\'{e}vy measure} \(\ell\) on \(\mathbb{R}^{G}\), thus \(\ell\) is a Borel \(\sigma\)-finite measure satisfying
\[\ell\left(\{\mathbf{0}\}\right)=0\text{ and }\int_{\mathbb{R}^{G}}\left|\Xi_{g}\left(\mathbf{x}\right)\right|^{2}\wedge1\,d\ell\left(\mathbf{x}\right)<+\infty\text{ for all }g\in G,\]
where \(\mathbf{0}=\left(0\right)_{g\in G}\). Pick also a constant {\bf drift}, \(b\in\mathbb{R}\). Now define
\[\mathbf{X}=\big(X_{g}\big)_{g\in G}\text{ by letting }X_{g}\coloneq I_{\ell}\left(\Xi_{g}\right)+b\text{ for }g\in G.\]
Then \(\mathbf{X}\) is an {\sc id}p stationary \(G\)-process. By virtue of the Maruyama representation (extending the L\'{e}vy--Khintchine representation of {\sc id}p vectors), every {\sc id}p stationary \(G\)-process is equal in distribution to a process constructed that way; see \cite[\S4]{maruyama1970infinitely}, \cite[Ch.~4]{sato1999levy}, \cite[\S II]{rajput1989spectral}, \cite[Prop.~2.10]{rosinki2018representations}.

\smallskip

In order to study ergodicity and weak mixing of \(\mathbf{X}\), it is important to realize the distribution of \(\mathbf{X}\) using an appropriate Poissonian action (for a full account of this construction see \cite[\S 3-4]{avraham2025poissonian}). Thus look at the measure preserving \(G\)-action \(\mathcal{R}:G\curvearrowright\left(\mathbb{R}^{G},\mathcal{B}^{G},\ell\right)\). The associated {\bf Poisson suspension} is the probability preserving \(G\)-action
\[\mathcal{R}^{\ast}:G\curvearrowright\big(\left(\mathbb{R}^{G}\right)^{\ast},\left(\mathcal{B}^{G}\right)^{\ast},\ell^{\ast}\big),\]
where the action \(\mathcal{R}^{\ast}=\{R_{g}^{\ast}\}_{g\in G}\) is defined by
\[R_{g}^{\ast}\left(\theta\right)\left(A\right)\coloneq\theta\left(R_{g}^{-1}\left(A\right)\right),\quad g\in G.\]

\smallskip

For every \(f\) whose Poisson stochastic integral \(I_{\ell}\left(f\right)\) as in \eqref{eq:plim} is defined, one has
\begin{equation}\label{eq:equivar}
I_{\ell}\left(f\circ R_{g}\right)=I_{\ell}\left(f\right)\circ R_{g}^{\ast}\quad\ell^{\ast}\text{-a.s. for every }g\in G.
\end{equation}
Indeed, it suffices to look at \(f=\mathbf{1}_{A}\) with \(0<\ell\left(A\right)<+\infty\), where by the \(R_{g}\)-invariance of \(\ell\),
\begin{align*}
I_{\ell}\left(\mathbf{1}_{A}\circ R_{g}\right)=I_{\ell}\big(\mathbf{1}_{R_{g}^{-1}\left(A\right)}\big)
&=N_{R_{g}^{-1}\left(A\right)}-\ell\left(R_{g}^{-1}\left(A\right)\right)\\
&=N_{A}\circ R_{g}^{\ast}-\ell\left(A\right)=I_{\ell}\left(\mathbf{1}_{A}\right)\circ R_{g}^{\ast}\quad\ell^{\ast}\text{-a.s.}
\end{align*}
Now if \(\mathbf{X}\) is any {\sc id}p stationary \(G\)-process, by virtue of the Maruyama representation there exists a L\'{e}vy measure \(\ell\) on \(\mathbb{R}^{G}\) and a drift \(b\in\mathbb{R}\), such that
\[\mathbf{X}=\left(X_{g}\right)_{g\in G}=\left(I_{\ell}\left(\Xi_{g}\right)+b\right)_{g\in G}.\]
Let the random variables \(X_{o}\coloneq I_{\ell}\left(\Xi_{e_{G}}\right)\). Since \(\Xi_{g}=\Xi_{e_{G}}\circ R_{g^{-1}}\) and by \eqref{eq:equivar}, we may write
\begin{equation}\label{eq:genfo}
\mathbf{X}=\left(X_{g}\right)_{g\in G}=\left(I_{\ell}\left(\Xi_{e_{G}}\circ R_{g^{-1}}\right)+b\right)_{g\in G}=\big(X_{o}\circ R_{g^{-1}}^{\ast}+b\big)_{g\in G}.
\end{equation}

\subsection{Proof of the main theorem}\label{sct:proof}

We first use \cref{thm:stochext} and \cref{lem:ergdens} to reduce the proof to the case of countable groups.

\begin{prop}\label{prop:reduct}
It suffices to prove the main theorem for countable (untopologized) groups.
\end{prop}

\begin{proof}[Proof of \cref{prop:reduct}]
Recall that, much like Gaussian random variables, a random variable which is a limit in probability of {\sc id}p random variables is itself {\sc id}p \cite[Lemma~7.8]{sato1999levy}. Therefore, when taking the stochastically continuous extension as in \cref{thm:stochext}, one generally also has that a stationary \(G\)-process \(\mathbf{X}\) is {\sc id}p if and only if its stochastically continuous extension \(\widehat{\mathbf{X}}\) is {\sc id}p. Consequently, by passing to the stochastically continuous extension and using the second part of \cref{thm:stochext}, it suffices to prove the main theorem for Polish groups under stochastic continuity. Now, by passing to a countable dense subgroup (which every Polish group possesses: take the subgroup generated algebraically by any countable dense subset) and restricting the process to this countable group of indices, we obtain an {\sc id}p stationary process indexed by a countable group. By \cref{lem:ergdens}, the ergodicity and weak mixing of this process are equivalent, respectively, to the ergodicity and weak mixing of the original {\sc id}p stationary process, and this readily implies that it suffices to prove the main theorem for countable groups.
\end{proof}

In light of \cref{prop:reduct}, for the rest of the discussion we assume \(G\) is a countable group. Let us recall two essential ingredients we use in our proof. The first is the characterization of ergodicity of Poisson suspensions, proved by Marchat \cite{marchat1978} for \(G=\mathbb{Z}\) (different proofs were given later in \cite[Thm.~1]{grabinsky1984poisson} and \cite[\S4]{roy2007ergodic}), and in \cite[Thm.~2]{avraham2025poissonian} for Polish groups (in particular countable).\footnote{Although it is stated there for nonatomic measures, it is true for general measures; see \cite[Rem.~5.9]{avraham2025poissonian}.} Thus, for every measure preserving \(G\)-action \(\mathcal{R}:G\curvearrowright\left(\mathbb{R}^{G},\mathcal{B}^{G},\ell\right)\) with Poisson suspension \(\mathcal{R}^{\ast}:G\curvearrowright\big(\left(\mathbb{R}^{G}\right)^{\ast},\left(\mathcal{B}^{G}\right)^{\ast},\ell^{\ast}\big)\), the following are equivalent:
\begin{enumerate}
    \item \(\mathcal{R}\) is {\bf null}: no \(A\in\mathcal{B}^{G}\) with \(0<\ell\left(A\right)<+\infty\) satisfies \(\ell\left(R_{g}^{-1}\left(A\right)\triangle A\right)=0\) for all \(g\in G\).
    \item \(\mathcal{R}^{\ast}\) is ergodic.
    \item \(\mathcal{R}^{\ast}\) is weakly mixing.
\end{enumerate}

\medskip

The second is the way ergodicity of an {\sc id}p stationary process reflects on its L\'{e}vy measure:

\begin{lem}\label{lem:null}
The L\'{e}vy measure of every ergodic non-constant {\sc id}p stationary \(G\)-process is null: it admits no invariant set of positive finite measure.
\end{lem}

\begin{proof}[Proof of \cref{lem:null}]
Let \(\mathbf{X}\) be a non-constant {\sc id}p stationary \(G\)-process with L\'{e}vy measure \(\ell\). Suppose there is \(A\in\mathcal{B}^{G}\) with \(0<\ell\left(A\right)<+\infty\) such that \(\ell\left(R_{g}^{-1}\left(A\right)\triangle A\right)=0\) for every \(g\in G\), and we prove \(\mathbf{X}\) is non-ergodic. By the Maruyama representation \eqref{eq:genfo}, we may assume that
\[\mathbf{X}=\left(X_{g}\right)_{g\in G}=\big(X_{o}\circ R_{g^{-1}}^{\ast}+b\big)_{g\in G},\text{ where }X_{o}=I_{\ell}\left(\Xi_{e_{G}}\right).\]
Define the L\'{e}vy measures \(\ell^{\prime}\left(\cdot\right)\coloneq\ell\left(\cdot\cap A\right)\) and \(\ell^{\prime\prime}\left(\cdot\right)\coloneq\ell\big(\cdot\cap A^{\complement}\big)\), put the random variables
\[X_{o}^{\prime}\coloneq I_{\ell^{\prime}}\left(\Xi_{e_{G}}\right)\,\text{ and }\,X_{o}^{\prime\prime}\coloneq I_{\ell^{\prime\prime}}\left(\Xi_{e_{G}}\right),\]
and consider the {\sc id}p stationary \(G\)-processes
\[\mathbf{X}^{\prime}\coloneq\big(X_{o}^{\prime}\circ R_{g^{-1}}^{\ast}\big)_{g\in G}\,\text{ and }\,\mathbf{X}^{\prime\prime}\coloneq\big(X_{o}^{\prime\prime}\circ R_{g^{-1}}^{\ast}+b\big)_{g\in G}.\footnote{If \(\ell\big(A^{\complement}\big)=0\) then \(\mathbf{X}^{\prime\prime}\) is a constant process. This does not hurt the proof.}\]
Since \(\ell=\ell^{\prime}+\ell^{\prime\prime}\) while the supports of \(\ell^{\prime}\) and \(\ell^{\prime\prime}\) are disjoint and \(G\)-invariant, we have that
\[\mathbf{X}\eqd\mathbf{X}^{\prime}\oplus\mathbf{X}^{\prime\prime}.\]
It follows from \cref{lem:sam} that, in order to show that \(\mathbf{X}\) is non-ergodic, it suffices to show that \(\mathbf{X}^{\prime}\) is non-ergodic. Then for the rest of the proof we will show that \(\mathbf{X}^{\prime}\) is non-ergodic, precisely because its L\'{e}vy measure, \(\ell^{\prime}\), has a finite positive total mass, \(\lambda\coloneq\ell^{\prime}\left(\mathbb{R}^{G}\right)=\ell\left(A\right)\).

\smallskip

Since \(\ell^{\prime}\) is finite, each of the two components of \(X_{o}^{\prime}=I_{\ell^{\prime}}\left(\Xi_{e_{G}}\right)\) as in \eqref{eq:plim} converges, and thus
\[X_{o}^{\prime}=\int_{\mathbb{R}^{G}}\Xi_{e_{G}}d\mathcal{N}_{\ell^{\prime}}-\int_{\mathbb{R}^{G}}\Xi_{e_{G}}\cdot\mathbf{1}_{\small\{ \small|\Xi_{e_{G}}\small|\leq1\small\}}d\ell^{\prime}\coloneq P_{o}-c_{o},\]
where \(P_{o}\) is an {\sc id}p random variable and \(c_{o}\) is a constant. Consider the measurable set
\[B\coloneq\left\{ \mathcal{N}_{\ell^{\prime}}\left(\mathbb{R}^{G}\right)=0\right\}\in\left(\mathcal{B}^{G}\right)^{\ast}.\]
Then \(B\) is \(\mathcal{R}^{\ast}\)-invariant (in fact, \(R_{g}^{\ast-1}\left(B\right)=B\) as sets for every \(g\in G\)), and \(\ell^{\prime\ast}\left(B\right)=e^{-\lambda}\in\left(0,1\right)\). Now on \(B\) it holds that \(P_{o}=\int_{\mathbb{R}^{G}}\Xi_{e_{G}}d\mathcal{N}_{\ell^{\prime}}=0\), hence \(X_{o}^{\prime}=-c_{o}\). Moreover, by the \(\mathcal{R}^{\ast}\)-invariance of \(B\), on \(B\) it holds that \(X_{o}^{\prime}\circ R_{g^{-1}}^{\ast}=-c_{o}\) for every \(g\in G\). Let \(\mathbf{x}_{o}=\left(-c_{o}\right)_{g\in G}\in\mathbb{R}^{G}\) and define the event \(E\coloneq\{\mathbf{X}^{\prime}=\mathbf{x}_{o}\}\). We clearly have \(\mathbb{P}_{\mathbf{X}^{\prime}}\left(R_{g}^{-1}\left(E\right)\triangle E\right)=0\) for every \(g\in G\). Since \(B\subseteq\mathbf{X}^{\prime-1}\left(E\right)\) we have
\[\mathbb{P}_{\mathbf{X}^{\prime}}\left(E\right)=\ell^{\prime\ast}\left(\mathbf{X}^{\prime-1}\left(E\right)\right)\geq\ell^{\prime\ast}\left(B\right)=e^{-\lambda}>0,\]
while \(\mathbb{P}_{\mathbf{X}^{\prime}}\left(E\right)<1\) since \(\mathbf{X}^{\prime}\) is non-constant (as \(\ell^{\prime}\) is non-zero). Therefore \(\mathbf{X}^{\prime}\) is non-ergodic.
\end{proof}

\begin{proof}[Final proof of the main theorem]
Suppose \(\mathbf{X}\) is an ergodic separable in probability {\sc id}p stationary \(G\)-process with L\'{e}vy measure \(\ell\). It is sufficient to assume that \(G\) is countable by \cref{prop:reduct}. By the Maruyama representation \eqref{eq:genfo}, we may assume that \(\mathbf{X}\) is of the form
\[\mathbf{X}=\left(X_{g}\right)_{g\in G}=\big(X_{o}\circ R_{g^{-1}}^{\ast}+b\big)_{g\in G}.\]
If \(X_{o}\equiv0\), then \(\mathbf{X}\) is the deterministic constant process \(\left(b\right)_{g\in G}\), and is trivially weakly mixing. Assume then that \(\mathbf{X}\)
is non-constant. Since \(\mathbf{X}\) is ergodic, \(\mathcal{R}:G\curvearrowright\left(\mathbb{R}^{G},\mathcal{B}^{G},\ell\right)\) is null by \cref{lem:null}. Then by the aforementioned \cite[Thm.~2]{avraham2025poissonian}, the Poisson suspension \(\mathcal{R}^{\ast}:G\curvearrowright\big(\left(\mathbb{R}^{G}\right)^{\ast},\left(\mathcal{B}^{G}\right)^{\ast},\ell^{\ast}\big)\) is weakly mixing. Consider the map
\[\left(\mathbb{R}^G\right)^{\ast}\to\mathbb{R}^{G},\quad
\theta\mapsto\big(X_{o}\big(R_{g^{-1}}^{\ast}\left(\theta\right)\big)+b\big)_{g\in G}.\]
By \eqref{eq:equivar} this map is \(G\)-equivariant. By virtue of the Maruyama representation, this map pushes \(\ell^{\ast}\) to \(\mathbb{P}_{\mathbf{X}}\). Therefore, the probability preserving \(G\)-system associated with \(\mathbf{X}\) is a factor of the weakly mixing probability preserving \(G\)-system \(\mathcal{R}^{\ast}:G\curvearrowright\big(\left(\mathbb{R}^{G}\right)^{\ast},\left(\mathcal{B}^{G}\right)^{\ast},\ell^{\ast}\big)\). Since every factor of a weakly mixing probability preserving \(G\)-system is weakly mixing, \(\mathbf{X}\) is weakly mixing.

\smallskip

It remains to treat the general infinitely divisible case. Let \(\mathbf{X}\) be an ergodic separable in probability infinitely divisible stationary \(G\)-process. Using the L\'{e}vy--Khintchine representation,
\[\mathbf{X}\eqd\mathbf{Z}\oplus\mathbf{P},\]
where \(\mathbf{Z}\) is a centered Gaussian stationary \(G\)-process and \(\mathbf{P}\) is an independent {\sc id}p stationary \(G\)-process. The uniqueness of the L\'{e}vy--Khintchine representation shows that the components are separable in probability as well. Since \(\mathbf{X}\) is ergodic, both \(\mathbf{Z}\) and \(\mathbf{P}\) are ergodic by \cref{lem:sam}. Now \(\mathbf{Z}\) is weakly mixing by the classical Gaussian result together with stochastically continuous extension, and \(\mathbf{P}\) is weakly mixing by the {\sc id}p case proved above. Hence \(\mathbf{Z}\otimes\mathbf{P}\) is weakly mixing. Since \(\mathbf{X}\) is a factor of \(\mathbf{Z}\otimes\mathbf{P}\) under the coordinate addition map, \(\mathbf{X}\) is weakly mixing.
\end{proof}

\section{Applications and the null condition}\label{sct:exm}

In what follows, we discuss examples of {\sc id}p stationary processes to which the main theorem applies. We will also suggest the null condition for weak mixing, which is formulated in terms of the underlying Maruyama representation when is available.

\smallskip

Let us start by mentioning that the main theorem is new already for {\sc id}p stationary processes indexed by general countable groups, and in particular for \(\alpha\)-stable stationary process indexed by nonamenable groups. In \cite[\S5]{mj2022mixing}, the authors used the Rosinski representation to construct \(\alpha\)-stable stationary processes indexed by non-elementary hyperbolic groups, which are never amenable, and proved they are ergodic \cite[Thm.~1.2]{mj2022mixing}. By the main theorem, these \(\alpha\)-stable stationary processes are weakly mixing.

\smallskip

The main theorem is formulated at the level of the law of an abstract {\sc id}p stationary process, without reference to a Maruyama representation as part of its data; that is, without presenting the process as a Poisson stochastic integral over a measure preserving system. In the proof, separability in probability allows us to transfer the problem to the Maruyama representation of some abstract countable indexing group, whose relation to the law of the original process is rather oblique. In the common case, however, the process is given from the outset by a Maruyama representation. We will show that when such a representation is given, the nullity of the representing system is a sufficient condition for weak mixing.

\smallskip

Let \(G\) be a group and \(\mathcal{T}=\{T_{g}\}_{g\in G}\) a measure preserving \(G\)-system on a standard \(\sigma\)-finite measure space \(\left(L,\mathcal{L},\ell\right)\). Recall that \(\mathcal{T}\) is {\bf null} if there does not exist a \(\mathcal{T}\)-invariant set \(A\) in \(L\) with \(0<\ell\left(A\right)<+\infty\). In particular, when \(\ell\) is an infinite measure, if \(\mathcal{T}\) is ergodic then it is null.

\smallskip

We shall use the Poisson stochastic integral and its properties presented in \cref{sct:Poissint}. While there it is presented for \(L=\mathbb{R}^{G}\) with \(G\) countable, the same applies to every standard measure space, and so we have a standard probability space \(\left(L^{\ast},\mathcal{L}^{\ast},\ell^{\ast}\right)\) realizing the Poisson point process with intensity \(\ell\); see \cite[\S3]{avraham2025poissonian}. Let \(f:L\to\mathbb{R}\) be a measurable function satisfying
\begin{equation}\label{eq:wedge}
\int_{L}\left|f\right|^{2}\wedge1\,d\ell<+\infty,
\end{equation}
and let \(I_{\ell}\left(f\right)\) be the Poisson stochastic integral with intensity \(\ell\), defined on \(\left(L^{\ast},\mathcal{L}^{\ast},\ell^{\ast}\right)\). Fix \(b\in\mathbb{R}\), and define a \(G\)-process \(\mathbf{X}=\left(X_{g}\right)_{g\in G}\) by
\begin{equation}\label{eq:genex}
X_{g}\coloneq I_{\ell}\left(f\circ T_{g^{-1}}\right)+b,\quad g\in G.
\end{equation}
Then \(\mathbf{X}\) is an {\sc id}p stationary \(G\)-process represented over the L\'{e}vy representation space \(\left(L,\mathcal{L},\ell\right)\). Note that \(\mathbf{X}\) is separable in probability: \(\{X_{g}:g\in G\}\) is a subset of the space \(L^{0}\left(L^{\ast},\mathcal{L}^{\ast},\ell^{\ast}\right)\), which is separable with respect to convergence in probability.

\begin{thm}\label{thm:nullrep}
If the measure preserving \(G\)-system \(\mathcal{T}\) is null, then the {\sc id}p stationary \(G\)-process \(\mathbf{X}\) defined in \eqref{eq:genex} is weakly mixing.
\end{thm}

\begin{proof}[Proof of \cref{thm:nullrep}]
We use the stochastically continuous extensions within \(\mathrm{Aut}\left(L,\mathcal{L},\ell\right)\), which is Polish in the weak topology; see \cite[\S2.1]{le2022polish}. Let
\[\widehat{G}\coloneq\overline{\{T_{g}:g\in G\}}\leq \mathrm{Aut}\left(L,\mathcal{L},\ell\right).\]
Then \(\widehat{G}\) is a Polish group, and the homomorphism \(\tau:G\to\widehat{G}\), \(\tau\left(g\right)=T_g\), has dense image. We now claim that a stochastically continuous extension \(\widehat{\mathbf{X}}\) of \(\mathbf{X}\) is given by
\[\widehat{X}_{T}\coloneq I_{\ell}\left(f\circ T^{-1}\right)+b=I_{\ell}\left(f\right)\circ\left(T^{-1}\right)^{\ast}+b,\quad T\in\widehat{G},\]
where we used \eqref{eq:equivar}. It is well-defined since each \(T\in\widehat{G}\) preserves \(\ell\). Moreover, for every \(g\in G\),
\[\widehat{X}_{\tau\left(g\right)}=I_{\ell}\left(f\circ T_{g}^{-1}\right)+b=I_{\ell}\left(f\circ T_{g^{-1}}\right)+b=X_{g},\]
so \(\widehat{\mathbf{X}}\mid_{\tau\left(G\right)}\eqd\mathbf{X}\). To see that \(\widehat{\mathbf{X}}\) is stochastically continuous, let \(T_{n}\to T\) in \(\mathrm{Aut}\left(L,\mathcal{L},\ell\right)\). By \cite[Prop.~3.1]{avraham2025poissonian} we have \(\left(T_{n}^{-1}\right)^{\ast}\to\left(T^{-1}\right)^{\ast}\) in \(\mathrm{Aut}\left(L^{\ast},\mathcal{L}^{\ast},\ell^{\ast}\right)\), and therefore
\[\widehat{X}_{T_{n}}=I_{\ell}\left(f\right)\circ\left(T_{n}^{-1}\right)^{\ast}+b\pto I_{\ell}\left(f\right)\circ\left(T^{-1}\right)^{\ast}+b=\widehat{X}_{T}.\]
Thus \(\widehat{\mathbf{X}}\) is a stochastically continuous extension of \(\mathbf{X}\) in the sense of \cref{thm:stochext}.

\smallskip

Since \(\mathcal{T}\) is null, the measure preserving \(\widehat{G}\)-action on \(\left(L,\mathcal{L},\ell\right)\) clearly remains null. Then by \cite[Thm.~2]{avraham2025poissonian}, the Poisson suspension over the measure preserving \(\widehat{G}\)-action on \(\left(L,\mathcal{L},\ell\right)\) is weakly mixing. As in the proof of the main theorem, \(\widehat{\mathbf{X}}\) is a factor of this Poisson suspension, hence \(\widehat{\mathbf{X}}\) is weakly mixing as a \(\widehat{G}\)-process. Finally, since \(\widehat{\mathbf{X}}\) is a stochastically continuous extension of \(\mathbf{X}\), the last part of \cref{thm:stochext} gives that \(\mathbf{X}\) is weakly mixing as a \(G\)-process.
\end{proof}

In the following we provide examples of {\sc id}p stationary processes of the general form \eqref{eq:genex}. We begin with \(\alpha\)-stable stationary processes, whose Maruyama representations are given over \emph{Maharam actions}, following the construction developed by one of the authors in \cite{roy2012maharam}.

\begin{exm}[Stable processes from Maharam actions]
Let \(M\) be a compact smooth manifold with \(\dim M>0\), and fix a volume form \(\mathrm{Vol}\) on \(M\). The \(C^{\infty}\)-diffeomorphism group \(\mathrm{Diff}\left(M\right)\) of \(M\) is a non-locally compact Polish group with the \(C^{\infty}\) compact-open topology. The action of \(\mathrm{Diff}\left(M\right)\) on \(M\) preserves only the measure class of \(\mathrm{Vol}\). Let \(J_{T}\) be the Jacobian of \(T\), and so
\[\mathrm{Vol}\left(T\left(A\right)\right)=\int_{A}J_{T}d\mathrm{Vol},\quad A\in\mathcal{B}\left(M\right).\]
Fix \(0<\alpha<2\), and consider the standard infinite measure space
\[\left(L,\mathcal{L},\ell\right)\coloneq\big(M\times\mathbb{R}_{>0},\mathcal{B}\left(M\times\mathbb{R}_{>0}\right),\mathrm{Vol}\left(dx\right)\otimes\frac{dt}{t^{1+\alpha}}\big),\]
where \(dt\) denotes the usual Lebesgue measure (so \(\frac{dt}{t}\) is the multiplicative Haar measure of \(\mathbb{R}_{>0}\)).

\smallskip

Let \(G\leq\mathrm{Diff}\left(M\right)\) be an abstract subgroup. There is an infinite measure preserving action \(G\curvearrowright\left(L,\mathcal{L},\ell\right)\), the associated Maharam action, given by
\[\widetilde{T}\left(x,t\right)\coloneq\big(T\left(x\right),tJ_{T}\left(x\right)^{1/\alpha}\big),\quad T\in G.\]
Given \(f\in L^{\alpha}\left(M,\mathrm{Vol}\right)\), define
\[F:M\times\mathbb{R}_{>0}\to\mathbb{R},\quad F\left(x,t\right)\coloneq tf\left(x\right).\]
Using the change of variable \(s=t\left|f\left(x\right)\right|\), one has
\[\int_{L}\left|F\right|^{2}\wedge1\,d\ell=\int_{M\times\mathbb{R}_{>0}}\left|tf\left(x\right)\right|^{2}\wedge1\,\frac{dt}{t^{1+\alpha}}d\mathrm{Vol}\left(x\right)=c_{\alpha}\cdot\int_{M}\left|f\right|^{\alpha}d\mathrm{Vol}<+\infty,\]
where
\[c_{\alpha}\coloneq\int_{0}^{+\infty}t^{2}\wedge1\,\frac{dt}{t^{1+\alpha}}<+\infty.\]
Fixing \(b\in\mathbb{R}\), we obtain an \(\alpha\)-stable stationary \(G\)-process \(\mathbf{X}\) by letting
\[X_{T}\coloneq I_{\ell}\big(F\circ \widetilde{T}^{-1}\big)+b,\quad T\in G.\]
The fact that the distribution of \(\mathbf{X}\) is \(\alpha\)-stable was shown in \cite[\S5.2]{roy2012maharam}.

\smallskip

Whenever \(G\) acts ergodically on \(\left(L,\mathcal{L},\ell\right)\) (as is the case \(G=\mathrm{Diff}\left(M\right)\); see \cite[App.~B]{avraham2025poissonian}), this action is in particular null, and hence \(\mathbf{X}\) is weakly mixing by \cref{thm:nullrep}.
\end{exm}

Next we give examples of {\sc id}p stationary processes indexed by groups which admit natural infinite measure preserving actions, but which need not carry any Polish group topology.

\begin{exm}
Consider the group \(\mathrm{Aut}\left(\mathbb{R},\lambda\right)\), where \(\lambda\) is the usual Lebesgue measure. With its weak topology, \(\mathrm{Aut}\left(\mathbb{R},\lambda\right)\) is a non-locally compact Polish group, with a tautological measure preserving action on \(\left(\mathbb{R},\lambda\right)\). Let \(G\leq\mathrm{Aut}\left(\mathbb{R},\lambda\right)\) be an abstract subgroup. Fix a measurable set \(A\) with \(0<\lambda\left(A\right)<+\infty\). Let \(N_{A}\) be the random variable which counts the points that lie in \(A\), and thus \(N_{A}\) has Poisson distribution with mean \(\lambda\left(A\right)\). Put \(b\coloneq\lambda\left(A\right)\) and define
\[X_{T}\coloneq I_{\lambda}\left(\mathbf{1}_{A}\right)\circ\left(T^{-1}\right)^{\ast}+b=N_{T\left(A\right)},\quad T\in G.\]
Here we used that \(I_{\lambda}\left(\mathbf{1}_{A}\right)=N_{A}-\lambda\left(A\right)\). Then \(\mathbf{X}\coloneq\left(X_{T}\right)_{T\in G}\) is an {\sc id}p stationary \(G\)-process. When the action of \(G\) on \(\left(\mathbb{R},\lambda\right)\) is null, \(\mathbf{X}\) is weakly mixing by \cref{thm:nullrep}.

\smallskip

Let us now discuss the role of topology in the preceding example. It may be the case that \(G\leq\mathrm{Aut}\left(\mathbb{R},\lambda\right)\) is not closed in the weak topology, but is still Polishable, namely it admits an intrinsic Polish group topology making the inclusion \(G\hookrightarrow\mathrm{Aut}\left(\mathbb{R},\lambda\right)\) continuous. For instance, this happens for full groups of countable equivalence relations with the uniform topology, and more generally for some orbit full groups with the topology of orbital convergence in measure; see \cite[Props.~4.7, 6.13; Thms.~5.7, 6.4, 6.18]{hoareau2025polish}. In this case, one sees that the resulting process is stochastically continuous.

\smallskip

The point of \cref{thm:nullrep} is that Polishability is not needed to deduce weak mixing. Define
\[\mathrm{Aut}\left(\mathbb{R},\lambda\right)_{\mathrm{f}}\coloneq\left\{ T\in\mathrm{Aut}\left(\mathbb{R},\lambda\right):\lambda\left(\left\{t\in\mathbb{R}:T\left(t\right)\neq t\right\}\right)<+\infty\right\}.\]
This subgroup is dense in \(\mathrm{Aut}\left(\mathbb{R},\lambda\right)\) in the weak topology, hence its action on \(\left(\mathbb{R},\lambda\right)\) is ergodic, and in particular null; see \cite[Ex.~4.6]{hoareau2025polish}. Nevertheless, \(\mathrm{Aut}\left(\mathbb{R},\lambda\right)_{\mathrm{f}}\) is not only non-Polishable, but admits no Polish group topology \cite[Thm.~2.6, 2.10]{le2022polish}.

\smallskip

More generally, let \(\mathbb{G}\leq\mathrm{Aut}\left(\mathbb{R},\lambda\right)\) be an ergodic full group (see \cite[\S4]{hoareau2025polish}), and denote
\[\mathbb{G}_{\mathrm{f}}\coloneq\mathbb{G}\cap\mathrm{Aut}\left(\mathbb{R},\lambda\right)_{\mathrm{f}}.\]
Then \(\mathbb{G}_{\mathrm{f}}\) is dense in \(\mathbb{G}\) in the weak topology, hence its action on \(\left(\mathbb{R},\lambda\right)\) is ergodic and in particular null. As for the topology, \(\mathbb{G}_{\mathrm{f}}\) admits a Polish group topology if and only if \(\mathbb{G}\) is the full group of a countable equivalence relation \cite[Thm.~6.27]{hoareau2025polish}. Therefore, whenever \(\mathbb{G}\) is not of this form, \(\mathbb{G}_{\mathrm{f}}\) gives a class of indexing groups which admit no Polish group topology. Nevertheless, the corresponding {\sc id}p stationary process
\[\mathbf{X}=\left(N_{T\left(A\right)}\right)_{T\in\mathbb{G}_{\mathrm{f}}}\]
is weakly mixing by \cref{thm:nullrep}.
\end{exm}

\appendix

\section{Weak mixing and double ergodicity}\label{app:wm}

It is a classical fact that a probability preserving transformation is doubly ergodic if and only if it is weakly mixing. Proofs of this equivalence appear in various settings (see, for instance, \cite[Thm.~1.6]{tempelman2013} and \cite[Thm.~2.1]{glasner2016weak}), typically via alternative characterizations, such as those based on the ergodic theorem or on Koopman representations, and often under additional assumptions, including invertibility, group actions, or measurability of the action. The equivalence in full generality was established in a private communication of E. Roy with \'{E}lise Janvresse and Thierry de la Rue. In the following, we give a direct and elementary proof.

\smallskip

A {\bf probability preserving transformation} of a probability space \(\left(\mathbb{X},\mathcal{X},\xi\right)\) is a measurable map \(T:\mathbb{X}\to\mathbb{X}\) such that \(\xi\left(T^{-1}\left(A\right)\right)=\xi\left(A\right)\) for every \(A\in\mathcal{X}\).\footnote{This is to distinguish from \emph{automorphism}, that is required to be invertible.} Let \(\Theta\) be a set. A {\bf probability preserving \(\Theta\)-system} on \(\left(\mathbb{X},\mathcal{X},\xi\right)\) is a collection \(\mathcal{T}=\{T_{\theta}\}_{\theta\in\Theta}\) of probability preserving transformations of \(\left(\mathbb{X},\mathcal{X},\xi\right)\). The diagonal product \(\mathcal{T}\times\mathcal{S}\) of probability preserving \(\Theta\)-systems \(\mathcal{T}\) and \(\mathcal{S}\) is defined in the same way as in the group setting, and so are the invariant \(\sigma\)-algebra \(\mathcal{I}_{\mathcal{T}}\), and the ergodicity of a probability preserving \(\Theta\)-system (see \cref{sct:ergwm}).

\begin{prop}\label{prop:appwm}
Let \(\Theta\) be a set and \(\mathcal{T}\) a probability preserving \(\Theta\)-system. Then \(\mathcal{T}\times\mathcal{T}\) is ergodic if and only if \(\mathcal{T}\times\mathcal{S}\) is ergodic for every ergodic probability preserving \(\Theta\)-system \(\mathcal{S}\).
\end{prop}

\begin{proof}[Proof of \cref{prop:appwm}]
One implication is simple: if \(\mathcal{T}\times\mathcal{S}\) is ergodic for every ergodic probability preserving \(\Theta\)-system, then \(\mathcal{T}\) is ergodic (by taking \(\mathcal{S}\) to be the trivial probability preserving \(\Theta\)-system), hence \(\mathcal{T}\times\mathcal{T}\) is ergodic. For the converse, suppose \(\mathcal{T}\) and \(\mathcal{S}\) are probability preserving \(\Theta\)-systems defined on \(\left(\mathbb{X},\mathcal{X},\xi\right)\) and \(\left(\mathbb{Y},\mathcal{Y},\eta\right)\), respectively. Assume \(\mathcal{T}\times\mathcal{T}\) and \(\mathcal{S}\) are both ergodic. Let \(f\in L^{\infty}\left(\mathbb{X}\times \mathbb{Y},\xi\otimes\eta\right)\) be a \(\mathcal{T}\times\mathcal{S}\)-invariant function,\footnote{All functions considered in this proof are real-valued.} and we may assume without loss of generality that \(\iint_{\mathbb{X}\times\mathbb{Y}}fd\xi\otimes\eta=0\). Define the function
\[F\in L^{\infty}\left(\mathbb{X}\times\mathbb{X},\xi\otimes\xi\right),\quad F\left(x,x'\right)\coloneq\int_{\mathbb{Y}}f\left(x,y\right)f\left(x',y\right)d\eta\left(y\right).\]
Since \(f\) is \(\mathcal{T}\times\mathcal{S}\)-invariant and \(\eta\) is \(\mathcal{S}\)-invariant, \(F\) is \(\mathcal{T}\times\mathcal{T}\)-invariant. Since \(\mathcal{T}\times\mathcal{T}\) is ergodic, \(F=C\) \(\xi\otimes\xi\)-a.e. for a constant \(C\) , and we claim that \(C=0\). To see this, look at the projection
\[f_{\mathbb{Y}}\in L^{\infty}\left(\mathbb{Y},\eta\right),\quad f_{\mathbb{Y}}\left(y\right)\coloneq\int_{\mathbb{X}}f\left(x,y\right)d\xi\left(x\right).\]
Since \(f\) is \(\mathcal{T}\times\mathcal{S}\)-invariant and \(\xi\) is \(\mathcal{T}\)-invariant, \(f_{\mathbb{Y}}\) is \(\mathcal{S}\)-invariant. Since \(\mathcal{S}\) is ergodic, \(f_{\mathbb{Y}}=c\) \(\eta\)-a.e. for a constant \(c\), and by Fubini's theorem we have
\[c=\int_{\mathbb{Y}} f_{\mathbb{Y}}d\eta=\iint\nolimits_{\mathbb{X}\times\mathbb{Y}}fd\xi\otimes\eta=0.\]
By another use of Fubini's theorem we obtain the desired claim:
\[C=\iint\nolimits_{\mathbb{X}\times\mathbb{X}}Fd\xi\otimes\xi=\int_{\mathbb{Y}}\Big(\int_{\mathbb{X}}f\left(x,y\right)d\xi\left(x\right)\Big)\Big(\int_{\mathbb{X}}f\left(x^{\prime},y\right)d\xi\left(x^{\prime}\right)\Big)d\eta\left(y\right)=\int_{\mathbb{Y}}f_{\mathbb{Y}}^{2}d\eta=0.\]

We now complete the proof by deducing that \(f=0\) \(\xi\otimes\eta\)-a.e. Define the operator
\[\Psi_{f}:L^{2}\left(\mathbb{X},\xi\right)\to L^{2}\left(\mathbb{Y},\eta\right),\quad\Psi_{f}\left(\psi\right)\left(y\right)\coloneq\int_{\mathbb{X}}f\left(x,y\right)\psi\left(x\right)d\xi\left(x\right).\]
Then by Fubini's theorem and since \(F=0\) \(\xi\otimes\xi\)-a.e., for every \(\psi\in L^{2}\left(\mathbb{X},\xi\right)\) we have
\[\left\Vert\Psi_{f}\left(\psi\right)\right\Vert _{L^{2}\left(\eta\right)}^{2}=\iint\nolimits_{\mathbb{X}\times\mathbb{X}}F\left(x,x^{\prime}\right)\psi\left(x\right)\psi\left(x^{\prime}\right)d\xi\otimes\xi\left(x,x^{\prime}\right)=0,\]
hence \(\Psi_{f}\equiv 0\). By Fubini's theorem once again, for all \(\psi\in L^{2}\left(\mathbb{X},\xi\right)\), \(\varphi\in L^{2}\left(\mathbb{Y},\eta\right)\) we have
\[0=\left\langle\Psi_{f}\left(\psi\right),\varphi\right\rangle _{L^{2}\left(\eta\right)}=\iint\nolimits_{\mathbb{X}\times\mathbb{Y}}f\left(x,y\right)\psi\left(x\right)\varphi\left(y\right)d\xi\otimes\eta\left(x,y\right)=\left\langle f,\psi\otimes\varphi\right\rangle _{L^{2}\left(\xi\otimes\eta\right)}.\]
As linear combinations of such \(\psi\otimes\varphi\) are dense in \(L^{2}\left(\mathbb{X}\times\mathbb{Y},\xi\otimes\eta\right)\), we obtain \(f=0\) \(\xi\otimes\eta\)-a.e.
\end{proof}

\section{A characterization of stochastic continuity}\label{app:Pursley}

The following theorem was proved by Pursley~\cite{pursley1977} for \(G=\mathbb{R}\). It is closely related to a classical theorem of Hoffmann-J\o rgensen~\cite{hoffmann1973existence} (see also~\cite[\S9.4]{samorodnitsky1994stable}), which deals with general processes but is restricted to processes taking values in a compact metric space. Here, we provide a version of Pursley's theorem for general metrizable semigroups, with proofs essentially due to Pursley.

\smallskip

Let \(G\) be a metrizable semigroup with right unit \(e_{G}\), thus \(ge_{G}=g\) for all \(g\in G\). Similarly to the group case, a \(G\)-process \(\mathbf{X}=\left(X_{g}\right)\) is {\bf stationary} if \(\left(X_{gh}\right)_{h\in G}\eqd\left(X_{h}\right)_{h\in G}\) for each \(g\in G\). Such a stationary \(G\)-process is {\bf stochastically continuous} if
\(X_{g}\pto X_{g_{o}}\) as \(g\to g_{o}\) for every \(g_{o}\in G\). For \(g\in G\), denote by \(Q_{g}\) the (potentially non-invertible) translations of \(\mathbb{R}^{G}\) given by
\(Q_{g}:\left(x_{h}\right)_{h\in G}\mapsto\left(x_{gh}\right)_{h\in G}\).\footnote{The collection \(\{Q_{g}\}_{g\in G}\) forms a right action of \(G\) on \(\mathbb{R}^{G}\).} In the following discussion, \(Q_{g}^{-1}\left(\cdot\right)\) stands only for preimage of sets.

\begin{thm}[Hoffmann-J\o rgensen/Pursley]\label{thm:PursleyA}
Let \(G\) be a metrizable semigroup. A stationary \(G\)-process \(\mathbf{X}\) is stochastically continuous if and only if the following functions are all continuous:
\begin{equation}\label{eq:contfunc}
G\to\left[0,1\right],\quad g\mapsto\mathbb{P}_{\mathbf{X}}\left(Q_{g}^{-1}\left(E\right)\triangle Q_{g_{o}}^{-1}\left(E\right)\right),\quad E\in\mathcal{B}^{G},\,g_{o}\in G.\footnote{When \(G\) is a group (as in \cref{thm:Pursley}), it suffices to verify \eqref{eq:contfunc} only for \(g_{o}=e_{G}\).}
\end{equation}
\end{thm}

\begin{lem}[Pursley]\label{lem:PursleyA}
Let \(G\) be a metrizable semigroup and \(\mathcal{T}=\{T_{g}\}_{g\in G}\) a probability preserving \(G\)-system defined on \(\left(\mathbb{X},\mathcal{X},\xi\right)\). For every \(g_{o}\in G\), the following family forms a \(\sigma\)-algebra:
\[\mathcal{C}_{g_{o}}\coloneq\big\{ E\in\mathcal{X}:\lim\nolimits_{g\to g_{o}}\xi\left(T_{g}^{-1}\left(E\right)\triangle T_{g_{o}}^{-1}\left(E\right)\right)=0\big\}.\]
\end{lem}

\begin{proof}[Proof of \cref{lem:PursleyA}]
It is clear that \(\mathbb{X}\in\mathcal{C}_{g_{o}}\). From the identity
\[T_{g}^{-1}\left(E\right)\triangle T_{g_{o}}^{-1}\left(E\right)=T_{g}^{-1}\left(\mathbb{X}\backslash E\right)\triangle T_{g_{o}}^{-1}\left(\mathbb{X}\backslash E\right),\]
it follows that \(\mathcal{C}_{g_{o}}\) is closed under complements. Note that for all \(E_{1},\dotsc,E_{n}\in\mathcal{X}\) we have
\[T_{g}^{-1}\big(\bigcap\nolimits_{i=1}^{n}E_{i}\big)\triangle T_{g_{o}}^{-1}\big(\bigcap\nolimits_{i=1}^{n}E_{i}\big)\,\subseteq\,\bigcup\nolimits_{i=1}^{n}T_{g}^{-1}\left(E_{i}\right)\triangle T_{g_{o}}^{-1}\left(E_{i}\right),\]
so \(\mathcal{C}_{g_{o}}\) is closed under finite intersections. Then it suffices to show closeness under countable \emph{disjoint} unions. Let \(E\coloneq\bigsqcup_{n\in\mathbb{N}}E_{n}\) with \(E_{n}\in\mathcal{C}_{g_{o}}\) for \(n\in\mathbb{N}\). Similarly to the above, we have
\[T_{g}^{-1}\left(E\right)\triangle T_{g_{o}}^{-1}\left(E\right)
\subseteq\bigcup\nolimits_{n\in\mathbb{N}}T_{g}^{-1}\left(E_{n}\right)\triangle T_{g_{o}}^{-1}\left(E_{n}\right),\]
and so
\[\xi\left(T_{g}^{-1}\left(E\right)\triangle T_{g_o}^{-1}\left(E\right)\right)\leq\sum\nolimits_{n\in\mathbb{N}}\xi\left(T_{g}^{-1}\left(E_n\right)\triangle T_{g_o}^{-1}\left(E_n\right)\right).\]
Since the system is probability preserving, the \(n^{\mathrm{th}}\) term of the series is bounded by \(2\cdot\xi\left(E_{n}\right)\), which forms a convergent series (since \(E_{n}\), \(n\in\mathbb{N}\), are disjoint). Then by dominated convergence
\[\lim\nolimits_{g\to g_{o}}\xi\left(T_{g}^{-1}\left(E\right)\triangle T_{g_{o}}^{-1}\left(E\right)\right)\leq\sum\nolimits_{n\in\mathbb{N}}\lim\nolimits_{g\to g_{o}}\xi\left(T_{g}^{-1}\left(E_{n}\right)\triangle T_{g_{o}}^{-1}\left(E_{n}\right)\right)=0.\qedhere\]
\end{proof}

\begin{proof}[Proof of \cref{thm:PursleyA}. (Pursley)]
Assume all functions \eqref{eq:contfunc} are continuous. Let some \(g_{o}\in G\). Given \(\epsilon>0\), pick a measurable
partition \(\mathbb{R}=\bigsqcup_{n\in\mathbb{N}}B_{n}\) with each \(B_{n}\) of diameter \(<\epsilon\), and define
\[E_{n}\coloneq\{\left(x_{h}\right)_{h\in G}:x_{e_{G}}\in B_{n}\}\in\mathcal{B}^{G},\quad n\in\mathbb{N}.\]
Then for all \(g\in G\) we have
\[\big\{\left|X_{g}-X_{g_{o}}\right|>\epsilon\big\}\subseteq
\bigcup\nolimits_{n\in\mathbb{N}}\big(\{X_{g}\in B_{n}\}\triangle\{X_{g_{o}}\in B_{n}\}\big)=\bigcup\nolimits_{n\in\mathbb{N}}\left(Q_{g}^{-1}\left(E_{n}\right)\triangle Q_{g_{o}}^{-1}\left(E_{n}\right)\right),\]
and therefore
\[\mathbb{P}\big(\left|X_{g}-X_{g_{o}}\right|>\epsilon\big)
\leq\sum\nolimits_{n\in\mathbb{N}}\mathbb{P}_{\mathbf{X}}\big(Q_{g}^{-1}\left(E_{n}\right)\triangle Q_{g_{o}}^{-1}\left(E_{n}\right)\big).\]
Using the stationarity of \(\mathbf{X}\), for each \(n\in\mathbb{N}\) we have
\[\mathbb{P}_{\mathbf{X}}\big(Q_{g}^{-1}\left(E_{n}\right)\triangle Q_{g_{o}}^{-1}\left(E_{n}\right)\big)\leq\mathbb{P}_{\mathbf{X}}\big(Q_{g}^{-1}\left(E_{n}\right)\big)+\mathbb{P}_{\mathbf{X}}\big(Q_{g_{o}}^{-1}\left(E_{n}\right)\big)=2\cdot\mathbb{P}_{\mathbf{X}}\left(E_{n}\right),\]
and \(\sum_{n\in\mathbb{N}}2\cdot\mathbb{P}_{\mathbf{X}}\left(E_{n}\right)=2\). Therefore, by dominated convergence in the above inequality,
\[\lim\nolimits_{g\to g_{o}}\mathbb{P}\big(\left|X_{g}-X_{g_{o}}\right|>\epsilon\big)\leq\sum\nolimits_{n\in\mathbb{N}}\lim\nolimits_{g\to g_{o}}\mathbb{P}_{\mathbf{X}}\big(Q_{g}^{-1}\left(E_{n}\right)\triangle Q_{g_{o}}^{-1}\left(E_{n}\right)\big)=0,\]
and so \(\mathbf{X}\) is stochastically continuous.

\smallskip

For the converse, assume \(\mathbf{X}\) is stochastically continuous. Fix \(g_{o}\in G\), and put
\[\mathcal{C}_{g_{o}}\coloneq\big\{E\in\mathcal{B}^{G}:\lim\nolimits_{g\to g_{o}}
\mathbb{P}_{\mathbf{X}}\big(Q_{g}^{-1}\left(E\right)\triangle Q_{g_{o}}^{-1}\left(E\right)\big)=0\big\}.\]
By \cref{lem:PursleyA}, \(\mathcal{C}_{g_{o}}\) is a sub-\(\sigma\)-algebra. It thus suffices to show
that \(\mathcal{C}_{g_{o}}\) contains a generating family in \(\mathcal{B}^{G}\). Let \(\mathcal{I}\) be the
family of bounded open intervals in \(\mathbb{R}\). For \(h\in G\) and \(B\in\mathcal{B}\left(\mathbb{R}\right)\) put
\[E_{h}^{B}\coloneq\big\{\left(x_{k}\right)_{k\in G}\in\mathbb{R}^{G}:x_{h}\in B\big\}.\]
Since \(\{E_{h}^{I}:h\in G,\,I\in\mathcal{I}\}\) generates \(\mathcal{B}^{G}\), we complete the proof by
showing that \(E_{h}^{I}\in\mathcal{C}_{g_{o}}\) for every \(h\in G\) and \(I\in\mathcal{I}\). Fix \(h\in G\)
and \(I\in\mathcal{I}\), and write
\[\mathbb{P}_{\mathbf{X}}\big(Q_{g}^{-1}\left(E_{h}^{I}\right)\triangle Q_{g_{o}}^{-1}\left(E_{h}^{I}\right)\big)
=\mathbb{P}_{\mathbf{X}}\big(Q_{g}^{-1}\left(E_{h}^{I}\right)\cap Q_{g_{o}}^{-1}\big(E_{h}^{\mathbb{R}\setminus I}\big)\big)+\mathbb{P}_{\mathbf{X}}\big(Q_{g}^{-1}\big(E_{h}^{\mathbb{R}\setminus I}\big)\cap Q_{g_{o}}^{-1}\left(E_{h}^{I}\right)\big),\]
and we must show that this converges to \(0\) as \(g\to g_{o}\). We will show that
\begin{equation}\label{eq:sufft}
\lim\nolimits_{g\to g_{o}}\mathbb{P}_{\mathbf{X}}\big(Q_{g}^{-1}\left(E_{h}^{I}\right)\cap Q_{g_{o}}^{-1}\big(E_{h}^{\mathbb{R}\setminus I}\big)\big)=0,
\end{equation}
and one similarly shows that \(\lim_{g\to g_{o}}\mathbb{P}_{\mathbf{X}}\big(Q_{g}^{-1}\big(E_{h}^{\mathbb{R}\setminus I}\big)\cap Q_{g_{o}}^{-1}\left(E_{h}^{I}\right)\big)=0\), hence \(E_{h}^{I}\in\mathcal{C}_{g_{o}}\). To show \eqref{eq:sufft}, write \(I=\left(a-\epsilon,a+\epsilon\right)\) and we may assume \(0<\epsilon<1\) and \(N_{\epsilon}\coloneq\lfloor 1/\epsilon\rfloor\geq1\). Decompose
\[I=\bigsqcup\nolimits_{n\geq N_{\epsilon}}B_{n},\text{ where }
B_{n}\coloneq\big\{b\in\mathbb{R}:\epsilon-1/n\leq\left|a-b\right|<\epsilon-1/\left(n+1\right)\big\},\]
and then for every \(g\in G\) we have
\[\mathbb{P}_{\mathbf{X}}\big(Q_{g}^{-1}\left(E_{h}^{I}\right)\cap Q_{g_{o}}^{-1}\big(E_{h}^{\mathbb{R}\setminus I}\big)\big)=\sum\nolimits_{n\geq N_{\epsilon}}
\mathbb{P}_{\mathbf{X}}\big(Q_{g}^{-1}\big(E_{h}^{B_{n}}\big)\cap Q_{g_{o}}^{-1}\big(E_{h}^{\mathbb{R}\setminus I}\big)\big).\]
Using the stationarity of \(\mathbf{X}\), for each \(n\geq N_{\epsilon}\) we have
\[\mathbb{P}_{\mathbf{X}}\big(Q_{g}^{-1}\big(E_{h}^{B_{n}}\big)\cap Q_{g_{o}}^{-1}\big(E_{h}^{\mathbb{R}\backslash I}\big)\big)\leq\mathbb{P}_{\mathbf{X}}\big(Q_{g}^{-1}\big(E_{h}^{B_{n}}\big)\big)=\mathbb{P}_{\mathbf{X}}\big(E_{h}^{B_{n}}\big),\]
and \(\sum_{n\geq N_{\epsilon}}\mathbb{P}_{\mathbf{X}}\big(E_{h}^{B_{n}}\big)\leq 1\). Then again by dominated convergence in the above inequality,
\begin{align*}
\lim\nolimits_{g\to g_{o}}\mathbb{P}_{\mathbf{X}}\big(Q_{g}^{-1}\left(E_{h}^{I}\right)\cap Q_{g_{o}}^{-1}\big(E_{h}^{\mathbb{R}\backslash I}\big)\big)
&\leq\sum\nolimits_{n\geq N_{\epsilon}}\lim\nolimits_{g\to g_{o}}\mathbb{P}_{\mathbf{X}}\big(Q_{g}^{-1}\big(E_{h}^{B_{n}}\big)\cap Q_{g_{o}}^{-1}\big(E_{h}^{\mathbb{R}\setminus I}\big)\big)\\
&\leq\sum\nolimits_{n\geq N_{\epsilon}}\lim\nolimits_{g\to g_{o}}\mathbb{P}\left(\left|X_{gh}-X_{g_{o}h}\right|>1/\left(n+1\right)\right)=0,
\end{align*}
where the last equality is by stochastic continuity of \(\mathbf{X}\). This establishes \eqref{eq:sufft}.
\end{proof}

\section*{Acknowledgments}

We thank Sasha Danilenko for sharing with us the idea of testing ergodicity of probability preserving systems via taking closures in the automorphism group, that inspired \cref{thm:stochext}. We are grateful to Zemer Kosloff for several helpful discussions. We thank Jeff Steif, from whom we learned about the canonical metric, for interesting discussions on infinite divisibility. We also thank the anonymous reviewer for their careful reading and thoughtful comments, which led to valuable additions and improved the clarity of the paper.

\bibliographystyle{amsplain}
\bibliography{references}

\end{document}